\documentclass{amsart}
\usepackage{amsmath,amsfonts,amssymb}
\theoremstyle{definition}
\newtheorem{rk}{Remark}[section]
\newtheorem{rema}[rk]{Remark}
\newtheorem{defi}[rk]{Definition}
\theoremstyle{plain}
\newtheorem{thm}{Theorem}[section]
\newtheorem{prop}[thm]{Proposition}
\newtheorem{cor}[thm]{Corollary}
\newtheorem*{claim}{Claim}
\newtheorem{coro}[thm]{Corollary}
\newtheorem{lemma}[thm]{Lemma}
\newtheorem{conj}{Conjecture}
\newtheorem*{mainthm}{Main Theorem}
\newtheorem{problem}[rk]{Problem}

\newcommand{\NN}{\mathbb N}

\newcommand{\diff}{{\operatorname{Diff}}}
\newcommand{\orb}{\operatorname{Orb}}
\newcommand{\D}{{\mathcal{D}}}
\newcommand{\C}{\mathcal{C}}

\newcommand{\cR}{{\mathcal{R}}}
\newcommand{\cG}{{\mathcal{G}}}
\newcommand{\cO}{{\mathcal{O}}}
\newcommand{\cB}{{\mathcal{B}}}
\newcommand{\QQ}{{\mathbb{Q}}}

\newcommand{\cU}{\mathcal{U}}
\newcommand{\F}{\mathcal{F}}
\newcommand{\mO}{\mathcal{O}}

\newcommand{\HT}{\overline{\rm HT}}

\newcommand{\La}{\Lambda}

\newcommand{\be}{\beta}
\newcommand{\ep}{\epsilon}
\newcommand{\de}{\delta}

\newcommand{\al}{\alpha}

 \def\NN{{\mathbb N}} 

\def\QQ{{\mathbb Q}} \def\RR{{\mathbb R}} 
\def\TT{{\mathbb T}}

 \def\ZZ{{\mathbb Z}}

\begin{document}

\title[Partial hyperbolicity and homoclinic tangencies] {Partial hyperbolicity and homoclinic tangencies}
\author{S. Crovisier, M. Sambarino, D. Yang}

\subjclass[2000]{}
\date{}
\keywords{}

\begin{abstract}
We show that any diffeomorphism of a compact manifold can be $C^1$ approximated by diffeomorphisms exhibiting a homoclinic tangency or by diffeomorphisms having a partial hyperbolic structure.
\end{abstract}

\maketitle

\section{Introduction}
\subsection{Characterization of partial hyperbolicity}
We are interested in describing the dynamics of a large class of diffeomorphisms
of a compact manifold $M$. This goal was achieved in a satisfactory way
about 40 years ago for the \emph{hyperbolic diffeomorphisms}.
These are the systems $f$ admitting
a filtration, i.e. a finite collection of open sets $\emptyset=U_0\subset U_1\subset \dots\subset U_s=M$ satisfying $f(\overline{U_i})\subset U_i$, such that the maximal invariant set $\Lambda_i$
in each level $U_i\setminus U_{i-1}$ is hyperbolic: there exists a splitting of the tangent bundle
into two invariant linear subbundles $T_{\Lambda_i}M=E^s\oplus E^u$
such that $E^s$ and $E^u$ are respectively strictly contracted and expanded by a forward iterate
of $f$.
These systems satisfy many good properties.
For instance, they possess a spectral decomposition: taking a finer filtration the sets $\Lambda_i$ are transitive. Moreover, they are stable under perturbations,
they can be coded,...

We are far from well understanding the dynamics beyond hyperbolicity
and one would like to extend the previous properties to weaker forms of hyperbolicity.
This paper deals with the following version of partial hyperbolicity.

\begin{defi}
An invariant set $\Lambda$ of a diffeomorphism $f$ is \emph{partially hyperbolic}
if its tangent bundle splits into invariant linear subbundles:
$$T_{\Lambda}M=E^s\oplus E^c_1\oplus\cdots\oplus E^c_k\oplus E^u,$$
such that each central bundle $E^c_i$ is one-dimensional
and such that considering a Riemannian metric on $M$
there exists a forward iterate $f^N$ which satisfies:
\begin{itemize}
\item[--] $\|Df^N.v\|\leq 1/2$ for each unitary $v\in E^s$ (one says $E^s$ is \emph{(uniformly) contracted}.),
\item[--] $\|Df^{-N}.v\|\leq 1/2$ for each unitary $v\in E^u$ (one syas $E^u$ is \emph{(uniformly) expanded}.),
\item[--] $\|Df^{N}_x.u\|\leq 1/2\|Df^N_x.v\|$ for each $x\in \Lambda$, each $i=0,\dots, k$
and each unitary vectors $u\in E^s\oplus\dots\oplus E_i^c$, $v\in E^c_{i+1}\oplus \dots\oplus E^u$ in $T_xM$.
\end{itemize}

A diffeomorphism $f$ is \emph{partially hyperbolic} if there exists a filtration
$\emptyset=U_0\subset U_1\subset \dots\subset U_s=M$
such that the maximal invariant set $\Lambda_i$ in each level $U_i\setminus U_{i-1}$
is partially hyperbolic.
\end{defi}
The type of the decomposition can be different on each piece $\Lambda_i$.
We allow $k=0$: in this case the piece is hyperbolic. We also allow
$E^s$ and/or $E^u$ to be trivial.
Different notions of partial hyperbolicity appear in the literature;
the decomposition of the central part into one-dimensional central bundles here
provides a good control of the tangent dynamics.
In particular, this gives a symbolic description and
the existence of equilibrium states for these systems, see the discussions in Section~\ref{ss.consequences}.
\medskip

The set of partially hyperbolic diffeomorphisms is open in the space $\diff^1(M)$
of diffeomorphisms endowed with the $C^1$-topology.
It is well known that there exist such diffeomorphisms that can not be
approximated by hyperbolic ones, see for instance~\cite{shub}.
Also partially hyperbolic diffeomorphisms are not dense
in $\diff^1(M)$ when $\dim(M)\geq 3$ (on surface this question is open):
from \cite{bd}, for an open set of diffeomorphisms
each filtration contains a piece $\Lambda_i$ whose tangent dynamics
has no non-trivial invariant continuous sub-bundle.
The lack of partial hyperbolicity in these examples is related
to the following obstruction.

\begin{defi}
A diffeomorphism $f:M\to M$ exhibits a \emph{homoclinic tangency} if there is a hyperbolic periodic orbit whose invariant manifolds $W^s(O)$ and $W^u(O)$ have a non transverse intersection.
\end{defi}
The homoclinic tangencies generate interesting dynamical instabilities,
see for instance~\cite{pt}.
Our goal is to provide a dichotomy between these two notions.

\begin{mainthm}
Any diffeomorphism $f$ can be approximated in $\diff^1(M)$ by diffeomorphisms which exhibit
a homoclinic tangency or by partially hyperbolic diffeomorphisms.
\end{mainthm}

This result thus gives a complete obstruction to the partial hyperbolicity
and decomposes the set of diffeomorphisms in two parts:
on one of them, the global dynamics is quite well described,
on the other one, the systems exhibit rich dynamical instabilities through local bifurcations.
One thus obtains an example of decomposition by \emph{phenomenon and mechanism} as discussed in~\cite{CP}.
\medskip

In fact, our result was motivated by a conjecture of Palis which proposes~\cite{palis} to characterize
the lack of hyperbolicity by homoclinic tangencies and heterodimensional cycles.
Several previous papers are related to our work.
\begin{itemize}
\item[--] \cite{PS} solved Palis conjecture on surfaces and implied
our result in this case.
\item[--] \cite{W} and \cite{gourmelon-tangence} showed that for diffeomorphisms far from tangencies,
the tangent bundle splits on the closure of the hyperbolic periodic orbits of a given stable dimension.
\item[--] The previous results imply the theorem for \emph{tame diffeomorphisms},
i.e. those which do not admit filtrations with an arbitrarily large number of non-trivial levels.
See~\cite{BDPR03,abcdw}.
For the same reason, the theorem holds in the conservative setting.
\item[--] \cite{GYW} gave a local version of the theorem:
far from homoclinic tangencies, any minimally non-hyperbolic sets is partially hyperbolic.
\item[--] \cite{Cro08} proved the theorem above
for diffeomorphisms that are not approximated
by diffeomorphisms exhibiting a heterodimensional cycle; this was used in
\cite{CP} for obtaining partial results on Palis conjecture.
\item[--] \cite{yang-aperiodic} obtained the partial hyperbolicity on aperiodic pieces
of generic diffeomorphisms far from homoclinic tangencies.
\end{itemize}

In view to~\cite{W}, one could hope to
generalize our result and obtain a positive answer to the following question.

\begin{problem}
Is it equivalent for a diffeomorphism to be not partially hyperbolic
and to be limit in $\diff^1(M)$ of diffeomorphisms exhibiting a homoclinic tangency?
\end{problem}
\medskip

\subsection{Precise statement}\label{ss.statement}
When one studies the global dynamics of a homeomorphism $f$ of a compact metric space $M$,
one splits the dynamics into pieces that can not be further decomposed by
filtrations. Let us define on $M$ the following relation: $x\sim y$
if and only if for any $\varepsilon>0$ there exists a periodic $\varepsilon$-pseudo-orbit
which contains both $x$ and $y$. Then,
the \emph{chain-recurrent set} $\cR(f)$ is the set of points $x\in M$ satisfying $x\sim x$.
On this set $\sim$ becomes an equivalence relation which defines a compact invariant
decomposition of $\cR(f)$ into its \emph{chain-recurrence classes}.
A \emph{chain-transitive set} of $f$ is an invariant compact set $K\subset M$
such that the restriction $f_{|K}$ has a single chain-recurrence class.

For any hyperbolic periodic orbit $O$ of a diffeomorphism, the closure of the transverse intersections
between the stable and the unstable manifolds of $O$ is called the
\emph{homoclinic class} of $O$ and denoted by $H(O)$.
In \cite{BC} it is proved that for a dense $G_\delta$ subset of $\diff^1(M)$
the chain-recurrence classes which contain periodic points are the homoclinic classes.
The other ones are called \emph{aperiodic classes}.

In the following we denote by $\HT\subset \diff^1(M)$
the closure of the set of diffeomorphisms exhibiting a homoclinic tangency.
Recall also that if $\mu$ is an invariant probability measure and if
$E$ is a continuous one-dimensional bundle over the support of $\mu$,
then one can define the \emph{Lyapunov exponent} of $\mu$ along $E$ as
$\int\log \|Df_{|E}\|d\mu$ by considering any riemannian metric on $M$.

\begin{thm}\label{main}
The diffeomorphisms $f$ in a dense $G_\delta$ subset $\cG\subset \diff^1(M)\setminus \HT$ has the following properties.
\begin{enumerate}
\item Any aperiodic class $\C$ is partially hyperbolic with a one-dimensional central bundle.
Moreover, the Lyapunov exponent along $E^c$ of any invariant measure supported on $\C$ is zero.
\item Any homoclinic class $H(p)$ has a partially hyperbolic structure
$$T_\C M=E^s\oplus E^c_1\oplus\ldots\oplus E^c_k\oplus E^u.$$
Moreover the minimal stable dimension of the periodic orbits of $H(p)$ is
$\dim(E^s)$ or $\dim(E^{s})+1$. Similarly the maximal stable dimension of the periodic orbits of $H(p)$ is
$\dim(E^s)+k$ or $\dim(E^{s})+k-1$. For every $i, 1\le i\le k$ there exist periodic points in $H(p)$ whose
Lyapunov exponent along $E^c_i$ is arbitrarily close to $0$.
\end{enumerate}
\end{thm}

One obtains the Main Theorem from Theorem~\ref{main} by proving that any diffeomorphism
$f\in \cG$ is partially hyperbolic. We first recall the following properties:
\begin{itemize}
\item[--] If a chain-recurrence class $\C$ is partially hyperbolic,
then the same holds on the maximal invariant set in a neighborhood of $\C$.
\item[--] Any limit of a sequence of chain-recurrence classes for the
Hausdorff topology is contained in a chain-recurrence class.
\end{itemize}
By compactness of the space of compact subsets of $M$,
one deduces from Theorem~\ref{main} that there exists a finite family of
open sets $V_1,\dots,V_\ell$ such that the maximal invariant set in each of them
is partially hyperbolic and any chain-recurrence class is contained in
one of the $V_i$. As a consequence considering a filtration $U_1,\dots,U_s$
such that the maximal invariant set in each level set
$U_{i}\setminus U_{i-1}$ is contained in one of the open sets $V_j$,
one gets the partial hyperbolicity of $f$. This gives the Main Theorem.
\bigskip

\begin{rema}
\begin{itemize}
\item[--] As we said, the first item was already obtained in~\cite{yang-aperiodic}.
\item[--] From~\cite{abcdw}, the set of stable dimensions of periodic orbits in a homoclinic class is an interval of
$\{0,\dots, \dim(M)\}$.
\item[--] For a homoclinic class $H(p)$ whose minimal stable dimension of periodic points
is $\dim(E^s)+1$, the bundle $E^{s}+E^c_1$ satisfies some weak hyperbolicity property. See Section~\ref{ss.central}.
\item[--] From~\cite[Section 2.4]{Cro08}, for aperiodic classes the extremal bundles
$E^s$ and $E^u$ are both non-degenerated.
\end{itemize}
\end{rema}

Theorem~\ref{main} answers partially~\cite[conjecture 1]{abcdw}, but other questions
remains: for diffeomorphisms as in Theorem~\ref{main}, are homoclinic classes
the only possible chain-recurrence classes?
does the minimal stable dimension of periodic orbits in a homoclinic class
coincide with $\dim(E^s)$? One would get positives answers if a spectral decomposition holds:

\begin{conj}[Bonatti~\cite{bonatti-panorama}]\label{c.finiteness}
There exists a dense $G_\delta$ subset in $\diff^1(M)\setminus \HT$
of diffeomorphisms having only finitely many chain-recurrence classes.
\end{conj}

One can also wonder if a local version of Theorem~\ref{main} holds for homoclinic classes:
\begin{problem}
Does there exists a dense $G_\delta$ subset of $\diff^1(M)$ of diffeomorphisms $f$
such that for any homoclinic class $H(O)$ one of the following holds:
\begin{itemize}
\item[--] either $H(O)$ is partially hyperbolic,
\item[--] or there exists a periodic orbit $O'\subset H(O)$
and diffeomorphisms arbitrarily close to $f$ in $\diff^1(M)$
exhibiting a homoclinic tangency associated to $O'$?
\end{itemize}
\end{problem}
\bigskip

\subsection{Linking periodic orbits to a class}\label{ss.link}
Let us explain the main difficulty for obtaining Theorem~\ref{main}:
we consider a chain-recurrence class $\C$ of a diffeomorphism $f$ which belongs to
a dense $G_\delta$ subset of $\diff^1(M)\setminus \HT$.

From~\cite{crovisier-approximation}, the class $\C$ is the Hausdorff limit
of a sequence of periodic orbits. From~\cite{W}, one gets an integer $N\geq 1$
and an invariant splitting $T_\C M=E\oplus F$ on $\C$ which is \emph{$N$-dominated}:
for each $x\in \C$ and each unitary vectors
$u\in E_x$, $v\in F_x$ one has $\|Df_x^N.u\|\leq 1/2\|Df_x^N.v\|$.

We thus have to analyze non-uniform
bundles and try to split them with one-dimensional central bundles.
From a selecting lemma of Liao, if a bundle $E$ is not uniform,
one can find in $\C$ an invariant compact subset $K$ with an invariant splitting of its tangent bundle of the form $T_KM=E^s\oplus E^c\oplus E^u$, where $E^c$ is one-dimensional and not
uniformly contracted nor expanded. By Ma\~n\'e's ergodic closing lemma,
one can find in any neighborhood of $K$ some periodic orbits $O^-, O^+$ of stable dimensions
$\dim(E^s)$ and $\dim(E^s)+1$ respectively.
If one can prove that there exist other periodic orbits $\widehat O^-,\widehat O^+$
with the same stable dimensions as $O^-,O^+$ and which approximate $\C$
in the Hausdorff topology, then one deduces again from \cite{W} that there exists
a splitting $T_\C M=E'\oplus E^c\oplus F'$ which coincides with $T_KM=E^s\oplus E^c\oplus E^u$
on $K$, giving a one-dimensional central bundle on $\C$ as required.
\medskip

One way to obtain this property is to prove that the periodic orbits $O^-,O^+$
are \emph{contained} in the class $\C$. Indeed by~\cite{BC} one would conclude
that $\C$ coincides with the homoclinic class $H(O^-)$ and $H(O^+)$.
Proving that $O^-$ and $O^+$ are contained in $\C$ is the main difficulty
in this subject, and so it is also for Palis conjecture.
In the case all the chain-recurrence classes are isolated in $\cR(f)$
(their number is finite, and the dynamics is said to be \emph{tame})
this problem does not appear; the same holds when the dynamics is conservative.

This difficulty is addressed by the following technical local result.
When $O$ is a periodic orbit of a diffeomorphism $f$ and $U$ is an open set containing $O$,
we define the \emph{local homoclinic class} $H(O,U)$ as the closure of the set
of transverse homoclinic orbits between the stable and the unstable manifolds of $O$
that are contained in $U$. It is a transitive invariant compact set contained in $\overline{U}$.

\begin{thm}\label{main2}
Let $f$ be a diffeomorphism in a dense $G_\delta$ subset of $\diff^1(M)$ and let $\La$ be a compact invariant
and chain-transitive set with a dominated splitting $T_\La M=E^s\oplus E^c\oplus F$. Assume also that
\begin{itemize}
\item[--] $E^s$ is uniformly contracted.
\item[--] $E^c$ is one dimensional and it is not uniformly contracted.
\item[--] There exists an ergodic measure $\mu$ supported on $\La$ whose Lyapunov exponent along $E^c$ is non-zero.
\end{itemize}
Then, one of the following holds:
\begin{enumerate}
\item For any neighborhood $U_0$ of $\Lambda$ , there exists a local homoclinic class
$H(O,U_0)$ containing $\Lambda$ where the stable dimension of $O$ equals $\dim E^s.$
\item For any $\theta >0$ and any neighborhood $U_0$ of $\La$, there exists a periodic orbit $O\subset U$ such that
\begin{itemize}
\item[--] The Lyapunov exponent of $O$ along $E^c$ belongs to $(-\theta,\theta)$
\item[--] $\La$ is contained in the local homoclinic class $H(O,U_0)$.
\end{itemize}
\end{enumerate}
\end{thm}

We would like to point out that the above theorem does not required the diffeomorphism to be far from homoclinic tangencies.
\bigskip

The goal to link a set having a non-uniform central bundle
to weak periodic orbits contained in a neighborhood is very similar to
Liao's and Ma\~n\'e's selecting lemmas (see for instance~\cite{wen-conjecture} or~\cite{mane-stabilite}).
For instance, with these technics~\cite{BGY09} proved the following:
\medskip

\emph{Let $f$ be a $C^1$-generic diffeomorphism and let $H(O)$ be a homoclinic class with dominated splitting $E^s\oplus F$ where $E^s$ is uniformly contracted, the stable dimension
of $O$ equals $\dim (E^s)$ and $F$ is not uniformly expanded.
Then, for any $\de$ there are periodic orbits in $H$ whose minimal Lyapunov exponent in $F$ belongs to $(0,\de).$}
\medskip

Our proof however is very different.
As in~\cite{Cro08}, it is obtained by analyzing the dynamics
along the central bundle and by using the central models introduced in~\cite{crovisier-palis-faible}.
The fact that we do not forbid heterodimensional cycles changes however
the philosophy and increases a lot the difficulty.
We need to develop two new strategies presented in Sections~\ref{s.connecting} and~\ref{s.trapped}.
\medskip

\paragraph{\bf Organization of the paper.}
The Section~\ref{pre} recalls genericity results and the Section~\ref{central}
presents the central models. The proof of Theorem~\ref{main2} from
Sections~\ref{s.connecting} and~\ref{s.trapped} is given in Section~\ref{s.main2}.
The proof of Theorem~\ref{main} (and hence of the Main Theorem)
appears in~Section~\ref{s.conclusion}. The consequences that are presented in section~\ref{ss.consequences}
are obtained in Section~\ref{s.consequences}.
Note that the first item of Theorem~\ref{main} about aperiodic classes
can be proved in a shorter way (independently from Sections~\ref{s.connecting} and~\ref{s.trapped}), see Section~\ref{ss.aperiodic}.
\bigskip

\subsection{Consequences. Dynamics far from homoclinic tangencies}\label{ss.consequences}
Our result allows to solve for diffeomorphisms far from homoclinic tangencies several problems
stated for $C^1$-diffeomorphisms. In the following,
we will say that a property holds for \emph{$C^1$-generic diffeomorphisms} if it is satisfied on
a dense $G_\delta$ subset of $\diff^1(M)$.
\medskip

\paragraph{\bf Symbolic extensions.}
The dynamics on partially hyperbolic sets where the central bundle has a dominated splitting into
one-dimensional subbundles have nice symbolic descriptions. This is obtained through the property of entropy
expansiveness, introduced by Bowen~\cite{Bow72}.

\begin{defi}
A homeomorphism $f$ of a compact metric space $X$ \emph{entropy-expansive}
if there exists $\varepsilon>0$ such that for any $x\in X$, the set
$$B(x,\varepsilon,f)=\{y\in X,\; \text{for any $n\in \ZZ$}, d(f^n(x),f^n(y))\leq \varepsilon\}$$
has zero entropy.
\end{defi}
Recall that a (non necessarily invariant) set $B\subset X$ has zero entropy when
the entropy of $B$ at any scale $\delta$ is zero:
let us denote by $N(B,\delta,f,n)$
the minimal number of points $x_1,\dots,x_N$ in $B$
so that any $x\in B$ stays $\delta$-close to some $x_i$ during $n$ iterations;
then $N(B,\delta,f,n)$ growths sub-exponentially with $n$ for any $\delta$.
\begin{rk} We provide some remarks from~\cite{Bow72}.
\begin{enumerate}
\item This notion is weaker than expansiveness, which assumes that the dynamical ball
$B(x,\varepsilon, f)$ is reduced to $x$.
\item In the definition, one can also replace $B(x,\varepsilon,f)$ by
the forward dynamical ball
$$B^+(x,\varepsilon,f)=\{y\in X,\; \text{for any $n\in \NN$}, d(f^n(x),f^n(y))\leq \varepsilon\}.$$
\item One can also introduce the ball
$$B_n(x,\varepsilon,f)=\{y\in X,\; \text{for any $0\leq k\leq n$}, d(f^k(x),f^k(y))\leq \varepsilon\}.$$
A map is entropy expansive if for $\varepsilon>0$ small the following quantity is zero
for any $x\in X$:
$$h^\ast_f(\varepsilon)=\lim_{\delta\to 0}\limsup_{n\to +\infty}\frac 1 n \log N(B_n(x,\varepsilon,f),\delta, f,n).$$
\item If $f$ is entropy expansive, then there exists $\varepsilon>0$ such that the topological entropy of $f$ coincides with the topological entropy at scale $\varepsilon$ defined as:
$$h_f(\varepsilon)=
\limsup_{n\to +\infty}\frac 1 n \log N(X,\varepsilon, f,n).$$
\item If the restriction of $f$ to the non-wandering set $\Omega(f)$ is entropy-expansive,
then so is $f$.
\end{enumerate}
\end{rk}
From the last remark, any entropy expansive system admits \emph{equilibrium states} (see~\cite[Proposition
2.19]{Bow75}): more precisely, for any continuous map $\varphi\colon X\to \RR$, there exists a measure that
realizes the maximum of the following quantity over all invariant probability measures $\mu$
$$P_f(\varphi)=\max_{\mu}\left\{h_f(\mu)+\int\varphi d\mu\right\},$$
where $h_f(\mu)$ denotes the entropy of $\mu$ for $f$.

In~\cite{bff} it is shown that any entropy-expansive system admits a \emph{principal
symbolic extension}, i.e. there exists an invariant compact set $Y$ of a full subshift
$(\{1,\dots,k\}^\ZZ,\sigma)$ and a continuous semi-conjugacy $\pi\colon Y\to X$
such that for any invariant measure $\mu$ of $(Y,\sigma)$, the entropy of $\mu$ for $\sigma$
and of $\pi_\ast(\mu)$ for $f$ are the same.

In~\cite{CoY05,dfpv}, it is shown that the restriction
of a diffeomorphism to a compact invariant set which is partially hyperbolic
and whose central bundle has a dominated splitting into one-dimensional subbundles is entropy-expansive. This result with the remark (5) above
and the Main Theorem implies the following.

\begin{cor}
Any diffeomorphism $f$ in an open and dense subset of ${\rm Diff}^1(M)\setminus\HT$
is entropy-expansive. In particular, it admits principal symbolic extensions.
Also, any continuous map $\varphi\colon M\to \RR$ has an equilibrium state.
\end{cor}

At the time we were finishing this paper, we learned that Liao, Viana and Yang \cite{lvy}
proved the entropy expansiveness for \emph{any} diffeomorphism $C^1$-far from homoclinic tangencies.
Their result uses the uniform dominated splitting over the support of ergodic measures for
diffeomorphism $C^1$-far from homoclinic tangencies. In our case, we do not have a uniform
version of Theorem~\ref{main} that holds for general diffeomorphisms in $\diff^1(M)\setminus \HT$.
\medskip

\paragraph{\bf Lack of hyperbolicity and weak periodic points.}
The stability of diffeomorphisms has been related to the hyperbolicity, see~\cite{mane-stabilite}.
This goal has been achieved by considering \emph{star diffeomorphisms}, i.e. systems whose
periodic orbits do not bifurcate by small $C^1$-perturbation.
We propose here to address a local version of this question.

For a homoclinic class $H(O)$ of $f$, if for any $\delta>0$, there
is a hyperbolic periodic point $q$ such that
\begin{itemize}
\item[--] $q$ is \emph{homoclinically related} with
$O$, i.e. its stable and unstable manifolds intersect transversally those of $O$,
\item[--] $q$ has one Lyapunov exponent in $(-\delta,\delta)$,
\end{itemize}
then one says that $H(O)$ contains \emph{weak
periodic orbits related to $p$}.

\begin{problem}
For generic $f\in{\rm Diff}^1(M)$, the homoclinic classes $H(O)$ satisfy:
\begin{itemize}
\item[--] either $H(O)$ is hyperbolic,
\item[--] or $H(O)$ contains weak periodic orbits related to $O$.
\end{itemize}
\end{problem}

Let us discuss how to solve this problem.
If the class $H(O)$ does not contain weak periodic orbits related to $O$,
then by a generic argument, one can prove that $H(O)$
admits a dominated splitting $T_{H(O)}M=E\oplus F$, where $\dim E$ equals to the stable dimension of $O$, and
one can get a uniform estimation for the contraction and the expansion at the period along the invariant bundles of periodic orbits related with $O$. (One can see \cite{YaG09}.)
By a variation of Ma\~n\'e's work, one can thus answer the problem when $E$ is uniformly contracted: \cite{BGY09} shows that $F$ is uniformly expanded in this case.

For diffeomorphisms far from homoclinic tangencies we have:

\begin{cor}\label{c.noweak}
For generic $f\in{\rm Diff}^1(M)\setminus\HT$,
any homoclinic class $H(O)$ satisfy:
\begin{itemize}
\item[--] either $H(O)$ is hyperbolic,
\item[--] or $H(O)$ contains weak periodic orbits related to $O$.
\end{itemize}
\end{cor}
\medskip

\noindent
\paragraph{\bf Palis conjecture.}
If $f$ has two hyperbolic periodic orbits $O_1$ and $O_2$ with different stable dimension
and whose invariant manifolds intersect (i.e.
such that $W^s(O_1)\cap W^u(O_2)$ and $W^s(O_2)\cap W^u(O_1)$ are non-empty),
then we say $f$ that has a \emph{heterodimensional cycle}.
One can generalize this definition: a \emph{robust heterodimensional cycle}
is a pair of transitive hyperbolic sets $K,L$ having different stable dimensions
such that for any diffeomorphism $C^1$-close to $f$
the stable and unstable sets of the continuations of $K$ and $L$ intersect.

\begin{conj}[Palis]
Every robustly non-hyperbolic diffeomorphism $f$ can be
accumulated by diffeomorphisms exhibiting either a homoclinic tangency or
a heterodimensional cycle.
\end{conj}

The following corollary of Theorem~\ref{main} was the main theorem of~\cite{Cro08}.

\begin{cor}\label{c.cycle}
For generic $f\in{\rm Diff}^1(M)$ which can not be approximated by diffeomorphisms exhibiting
a homoclinic tangency or a heterodimensional cycle, any homoclinic class $H(O)$ admits a partially
hyperbolic splitting:
$$T_{H(0)}M=E^s\oplus E^c_1\oplus E^c_2\oplus E^u,$$
where $\dim E^c_i=0~{\rm or} ~1$, for $i=1,2$ and $\dim(E^s\oplus E^c_1)$ coincides with the stable dimension of $O$.
\end{cor}
\medskip

\noindent
\paragraph{\bf Lyapunov stable classes.}
One sometimes focuses the study of a system to its attractors, since
they can describe the forward behavior of most orbits.
We will say that a chain-recurrence class is \emph{Lyapunov stable}
if it admits arbitrarily small neighborhoods $U$ satisfying $f(U)\subset U$.
It is known~\cite{BC} that for $C^1$-generic diffeomorphisms,
the forward orbit of points in a dense $G_\delta$ subset of $M$
accumulates on a Lyapunov stable class.

In restriction to these classes, Palis conjecture holds~\cite{CP}.
Combining this result with ours we get the following
statement, which was proposed to us by Yi Shi.

\begin{cor}\label{c.lyap-stable}
For generic $f\in \diff^1(M)\setminus \HT$, if $\C$ is a Lyapunov stable chain-recurrence class, then
\begin{itemize}
\item[--] either it is hyperbolic,
\item[--] or it contains a robust heterodimensional cycle.
\end{itemize}
\end{cor}

For generic diffeomorphisms in $\diff^1(M)\setminus \HT$,
more is known about Lyapunov stable chain-recurrence classes.
By~\cite{yang-lyapunov} they are homoclinic classes
and by~\cite{bgly} they are \emph{residual attractors},
i.e. for each of them there exists a neighborhood $U$
such that the forward orbit of any point in a dense $G_\delta$ subset of $U$
accumulates on the class.

When $M$ is connected, the bi-Lyapunov stable classes of a generic diffeomorphism,
i.e. the chain-recurrence classes that are Lyapunov stable both for $f$ and $f^{-1}$,
coincide with the whole manifold, see~\cite{rafael} and previous works~\cite{ABD,rafael-martin}.
In particular, if a chain-recurrence class has non-empty interior, $M$ is transitive.
\medskip

\noindent
\paragraph{\bf Bound on the number of classes.}
The conjecture~\ref{c.finiteness} asserts that
for generic diffeomorphism $f\in\diff^1(M)\setminus\HT$
the number of chain-recurrence classes is finite.
From the Main Theorem, the classes are partially hyperbolic and
the number of central bundles may be large.
One can bound the number of classes having a central bundle of dimension larger or equal to $2$.

\begin{cor}\label{c.centraldimension}
For generic $f\in{\rm Diff}^1(M)\setminus \HT$,
there exist only finitely many chain-recurrence classes having more than
one non-uniform one-dimensional central bundle.
\end{cor}
\medskip

\paragraph{\bf Index completeness.}
As we explained in section~\ref{ss.link}, it is sometimes important to know
the stable dimensions of the periodic orbits contained in a given neighborhood
of a set $K$ or which approximate a set for the Hausdorff topology.
\cite{abcdw} discussed this problem of \emph{index completeness} for chain
transitive sets of generic diffeomorphisms and proved that
homoclinic classes have some ``inner-index-completeness'' property.

For a chain-transitive set $\Lambda$ of $f$, one defines ${\rm ind}(\Lambda)$
as the set of integers $i$ such that $\Lambda$ is the Hausdorff limit of a sequence
of hyperbolic periodic orbits of stable dimension $i$.
If ${\rm ind}(\Lambda)$ is an interval of $\ZZ$, then one says that
$\Lambda$ is \emph{index complete}.
\begin{problem}
Are chain-transitive sets of $C^1$-generic $f$ index complete?
\end{problem}
We obtain a positive answer far from homoclinic tangencies.
\begin{cor}\label{c.index-completeness}
For generic $f\in \diff^1(M)\setminus \HT$ and for any chain-transitive set $\Lambda$,
\begin{enumerate}
\item $\Lambda$ is index complete.
Moreover, $\Lambda$ admits a partially hyperbolic splitting
$T_\Lambda=E^s\oplus E^c_1\oplus \cdots\oplus E^c_k\oplus E^u$, where
$\dim(E^c_i)=1$
and ${\rm ind}(\Lambda)=\{\dim E^s,\dots, \dim E^s+k\}$.
\item There exists a neighborhood $U$ of $\Lambda$ such that
$\operatorname{ind}(\Lambda)$ coincides with the set of stable dimensions
of the hyperbolic periodic orbits contained in $U$.
\end{enumerate}
\end{cor}
\medskip

\paragraph{\bf Ergodic closing lemma inside homoclinic classes.}
The ergodic closing lemma of Ma\~n\'e~\cite{mane-ergodic}
asserts that for $C^1$-generic diffeomorphisms any ergodic probability measure $\mu$
is the weak limit of a sequence of periodic orbits $(O_n)$.
As focused in section~\ref{ss.link}, it is important to know if in some cases
one can choose the periodic orbits $O_n$ in the chain-recurrence class supporting $\mu$.
This corresponds to~\cite[Conjecture 2]{bonatti-panorama}.

We obtained a positive answer in a very particular case.

\begin{cor}\label{c.ergodic}
For any generic diffeomorphism $f\in{\rm Diff}^1(M)\setminus\HT$,
let $H(p)$ be a homoclinic class and let $i$ be the minimal stable dimension of its
periodic orbits.
Then, for any ergodic measure $\mu$ supported on $H(p)$
and whose $i^\text{th}$ Lyapunov exponent is zero,
there exists periodic orbits $O_n$ contained in $H(p)$
whose associated measures converge for the weak topology towards the measure $\mu$.
\end{cor}
\medskip

\section{Preliminary definitions and results}\label{pre}

For the sake of completeness in this section we recall some definitions and results that are useful in our context. The reader may skip this section and return to it when referenced.

\subsection{General definitions} Let $f$ be a diffeomorphism of $M$.

The positive and negative orbits of a point $\{f^{n}(x),n\geq 0\}$
and  $\{f^{-n}(x),n\geq 0\}$ are denoted by $\operatorname{orb}^+(x)$ and
$\operatorname{orb}^-(x)$.

The \emph{chain-stable set} $W^{ch-s}_U(K)$ of a (not necessarily $f$-invariant) compact set $K$
inside a set $U$ that contains $K$
is the set of points $x\in M$ that can be joint to a point of $K$ by an $\varepsilon$-pseudo-orbit
contained in $U$ for any $\varepsilon>0$.
When $U=M$ we denotes $W^{ch-s}(M)=W^{ch-s}_U(K)$.

The \emph{local stable set} $W^s_\eta(x)$ of size $\eta>0$ of a point $x\in M$
is the set of points $y\in M$ such that the distance between $f^n(x)$ and $f^n(y)$
for $n\geq 0$ remains smaller than $\eta$ and goes to $0$ as $n$ goes to $+\infty$.

The (stable) \emph{index} of a hyperbolic periodic point is its stable dimension.

Two hyperbolic periodic orbits $O_1$ and $O_2$ contained in an open set $U$
are \emph{homoclinically related in $U$} if there exist transverse intersection points
in $W^s(O_1)\cap W^u(O_2)$ and $W^u(O_1)\cap W^s(O_2)$ whose orbits are contained in $U$.
This implies that $H(O_1,U)=H(O_2,U)$.

\subsection{Genericity results}
We then recall some properties that hold for diffeomorphisms in
a dense $G_\delta$ subset of $\diff^1(M)$.
\medskip

\paragraph{\bf Weak periodic points.}
Franks lemma~\cite{franks} allows to change the index of a periodic orbit having a weak Lyapunov exponent
by a $C^1$-small perturbation. With a classical Baire argument, one gets the following.
\begin{lemma}\label{l.index}
Let $f$ be a diffeomorphism in a dense $G_\delta$ subset of $\diff^1(M)$
and consider a sequence of hyperbolic periodic orbits $(O_n)$ of stable index $i>0$ which converges for the Hausdorff topology towards a set $\Lambda$. Assume also that the $i^\text{th}$ Lyapunov exponent $\lambda_i(O_n)$ of $O_n$
goes to zero as $n\to \infty$. Then there exist hyperbolic periodic orbits of index $i-1$ arbitrarily close to
$\Lambda$ for the Hausdorff topology whose $i^\text{th}$ Lyapunov exponent is arbitrarily close to zero.
\end{lemma}
\medskip

\paragraph{\bf Connecting lemmas.}

We need the following lemmas for connecting orbits.

\begin{thm}\cite{hayashi}\label{l.connecting}
For any $f$ and any $C^1$ neighborhood $\mathcal U$ of $f$, there is $L=L(\mathcal U)\in\NN$ such that any non-periodic point $z\in M$ admits two arbitrarily small neighborhoods $B_z\subset \hat B_z$ with the following property.

Consider two points $x,y$ outside $U_{L,z}:=\cup_{i=0}^{L-1} f^i(\hat B_z)$ having iterates
$f^{n_x}(x)$ and $f^{-n_y}(y)$ in $B_z$ such that $n_x,n_y\geq 1$.
Then, there exists $g\in \cU$ such that:
\begin{enumerate}
\item[--] $g=f$ on $M\setminus U_{L,z}$,
\item[--] $y$ belongs to the forward orbit of $x$ for $g$: there is $N\leq n_x+n_y$ such that $g^N(x)=y$,
\item[--] the orbit segment $\{x,g(x),\dots,g^N(x)\}$ is contained in the union of the orbit segments
$\{x,f(x),\dots,f^{n_x}(x)\}$, $\{f^{-n_y}(y),\dots,y\}$ and of $U_{L,z}$.
\end{enumerate}
\end{thm}

This allows to compare local homoclinic classes (the argument is similar to~\cite[Lemma 4.2]{GaW03}).

\begin{coro}\label{c.class}
Let $f$ be a diffeomorphism in a dense $G_\delta$ subset of $\diff^1(M)$,
let $p,q$ be two hyperbolic periodic points and let $U,V$ be two open sets satisfying $\overline U\subset V$.
Then:
\begin{itemize}
\item[--] If $H(p,U)$ contains $q$ and if $p,q$ have the same index, the orbits of $p,q$ are homoclinically related in $V$.
\item[--] If $H(p,U)$ and $H(q,U)$ intersect, $H(p,V)$ contains $H(q,U)$.
\end{itemize}
\end{coro}

The connecting lemma for pseudo-orbits proved in~\cite{BC} can be restated in the following way
(see~\cite{crovisier-approximation}):

\begin{thm}\cite{BC}\label{t.weaktransitive}
Let $f$ be a diffeomorphism in a dense $G_\delta$ subset of $\diff^1(M)$,
let $K$ be a compact set, $x,y$ be two points in $K$ such that $y\in W^{ch-u}_K(x)$.
Then, for any neighborhoods $U_x$ of $x$, $U_y$ of $y$ and $U_K$ of
$K$, there are $z\in U_x$ and $n\in\NN$ such that $f^n(z)\in U_y$ and $\{z,f(z),\cdots,f^n(z)\}\subset U_K$.
\end{thm}

Using Theorem~\ref{t.weaktransitive}, one can improve Corollary~\ref{c.class}
about local homoclinic classes and get the following (see also~\cite[Section 1.2.3]{BC}).
\begin{coro}\label{c.local-class}
Let $f$ be a diffeomorphism in a dense $G_\delta$ subset of $\diff^1(M)$.
Then, for any hyperbolic periodic point $p$, for any compact sets $\Lambda, \Delta$
such that $\Lambda$ is contained in the local chain-stable and chain-unstable sets of $p$ inside $\Delta$,
for any neighborhood $U$ of $\Delta$, the local homoclinic class $H(p,U)$ contains $\Lambda$.
\end{coro}

In particular, the homoclinic classes of a $C^1$-generic diffeomorphism are chain-recurrence classes.
Also considering a hyperbolic periodic orbit $O$ and two chain-transitive compact sets $K\subset \La$ contained in an open set $U$ such that $K\subset H(O,U)$, the class $H(O,V)$ contains $\La$ for any open neighborhood
$V$ of $\overline U$.
\medskip

The following global connecting lemma allows to control the support of orbit segments.

\begin{thm}\cite{crovisier-approximation}\label{t.hausdorff}
For any diffeomorphism $f$ in a dense $G_\delta$ subset of $\diff^1(M)$ and for any $\delta>0$,
there exists $\varepsilon>0$ such that for any $\varepsilon$-pseudo-orbit segment
$X=\{z_0,z_1,\dots,z_n\}$, one can find an orbit segment $O=\{x,f(x),\dots,f^m(x)\}$
which is $\delta$-close to $X$ for the Hausdorff distance.

Moreover if $X$ is a periodic pseudo-orbit (i.e. $z_n=z_0$), then the point $x$ can be chosen $m$-periodic.
In particular chain-transitive sets are Hausdorff limit of periodic orbits.
\end{thm}

\medskip

\paragraph{\bf Dominated splitting far from homoclinic tangencies.}
The existence of dominated splitting far from homoclinic tangencies was obtained in \cite{PS} in the two dimensional setting and by \cite{W} in the general case. The basic idea is that periodic orbits may have at most one Lyapunov exponent close to zero.

\begin{thm}\cite{W,wen-conjecture}\label{wen}
Let $f\in \diff^1(M)\setminus\HT$ and let $i\in\{1,\ldots,\dim M-1\}.$ Then, the tangent bundle above the set of hyperbolic periodic points with index $i$ has a dominated splitting $E\oplus F$ such that $\dim (E)=i.$bb

Moreover, there exist $\de,C>0$, $\sigma\in (0,1)$ and an integer $N\geq 1$ such that if $\mO$ is a hyperbolic periodic orbit of period $\tau$ and denoting by $E\oplus E^c\oplus F$ the splitting over $T_\mO M$ into the characteristic spaces whose Lyapunov exponents belong to $(-\infty,\de], (-\de,\de), [\de,\infty)$, then
\begin{itemize}
\item[--] $E^c$ has at most dimension one,
\item[--] the splitting is $N$-dominated,
\item[--] for any point $x\in \mO$ we have
\begin{equation}\label{e.wen}
\prod_{k=0}^{[\tau/N]-1}\|Df_{|E}(f^{kN}(x))\|\leq C\sigma^\tau, \quad
\prod_{k=0}^{[\tau/N]-1}\|Df_{|F}^{-1}(f^{-kN}(x))\|\leq C\sigma^\tau.
\end{equation}
\end{itemize}
\end{thm}

The last result induces dominated splitting over homoclinic classes.
\begin{cor}\label{weak}
For any $f$ in a dense $G_\delta$ subset of $\diff^1(M)\setminus \HT$ there exists $\delta$
such that any homoclinic class $H(O)$ satisfies the following.

If $O$ has no Lyapunov exponent in $(-\de,\de)$,
then, $H(O)$ has a dominated splitting $E\oplus F$ where $\dim(E)$ coincides with the stable index of $O$.

If $O$ has a Lyapunov exponent in $(-\de,\de)$, then, $H(O)$ has a dominated splitting $E\oplus E^c\oplus F$ where $E^c$ is one dimensional and the Lyapunov exponent of $O$ along $E^c$ belongs to $(-\de,\de).$
\end{cor}
\begin{proof}
The set of hyperbolic periodic points with Lyapunov exponents close to those of $O$ are dense in $H(O)$
by~\cite[Lemma 4.1]{bcdg}. By the above Theorem the result follows.
\end{proof}

The following is a consequence from Theorem~\ref{wen} and Liao's selecting lemma.
The proof is exactly the same as in~\cite[Theorem 1]{Cro08}, the statement we give now is more local.
\begin{thm}\cite[Theorem 1]{Cro08}\label{C3options}
Let $f\in \diff^1(M)\setminus \HT$ and let $K_0$ be an invariant compact set having dominated splitting $E\oplus F.$ If $E$ is not uniformly contracted, then for any neighborhood $U$ of
$K_0$ one of following occurs:
\begin{enumerate}
\item\label{lessindex} $K_0$ intersects a local homoclinic class $H(O,U)$
associated to a periodic orbit $O$ whose stable index is strictly less than $\dim(E).$
\item\label{equalindex} $K_0$ intersects local homoclinic classes $H(O_n,U)$
associated to a periodic orbit $O_n$ whose stable index is equal to $\dim(E)$ and which contains weak periodic orbits: for any $\de>0$ there exists a sequence of hyperbolic periodic orbits $\mO_n$ that are homoclinically related together in $U$, that converge for the Hausdorff topology toward a compact subset $K$ of $K_0,$ whose stable index are equal to $\dim E$ and whose maximal exponent along $E$ belongs to $(-\de,0).$
\item\label{subset} There exists an invariant compact set $K\subset K_0$ which has partially hyperbolic structure $E^s\oplus E^c\oplus E^u$. Moreover, $\dim E^s<\dim E$ the central bundle $E^c$ is one dimensional and any invariant measure supported on $K$ has a Lyapunov exponent along $E^c$ equal to $0.$
\end{enumerate}
\end{thm}

Theorems~\ref{wen} and~\ref{t.weaktransitive} allow to extend the dominated splitting over a chain-transitive set
to larger sets.
\begin{prop}\cite[Proposition 1.10]{Cro08}\label{spread}
Let $f$ be a diffeomorphism in a dense $G_\de$ set  $\cG_{Ext}\subset \diff^1(M)\setminus\HT$ and let $K$ be chain-transitive set which has a partially hyperbolic splitting $E^s\oplus E^c\oplus E^u$ such that $E^c$ is one-dimensional and such that any invariant measure supported on $K$ has a Lyapunov exponent along $E^c$ equal to $0.$ Then, for any chain-transitive set $A$ that strictly contains $K$ there exists a chain-transitive set $A'\subset A$ that strictly contains $K$ and there exists a dominated decomposition $E_1\oplus E^c\oplus E_3$ on $A'$ that extends the partially hyperbolic structure of $K$. Moreover, the set $A'$ is the Hausdorff limit of periodic orbits of stable index $\dim (E^s)$ whose Lyapunov exponent along $E^c$ is arbitrarily close to zero.

\end{prop}
\medskip

\paragraph{\bf Indices of a homoclinic class.}
We now state a result from \cite[corollaries 2 and 3]{abcdw}:

\begin{thm}\cite[corollaries 2 and 3]{abcdw}\label{indices}
Let $f$ be a diffeomorphism in a dense $G_\delta$ subset of $\diff^1(M)\setminus\HT$ and let $H$ be a homoclinic class having hyperbolic
saddles of stable indices $\al$ and $\be$ with $\al<\be.$ Then:
\begin{enumerate}
\item $H$ contains periodic points of stable index $j$ for every $\al\le j\le \be.$
\item For every $j\in \{\alpha+1,\dots,\beta\}$ there exist periodic orbits in $H$
of stable index $j-1$ and $j$
whose $j^\text{th}$-Lyapunov exponent is arbitrarily close to $0$.
\end{enumerate}

\end{thm}

\begin{rk}\label{r.indices}
The last theorem is also valid for local homoclinic classes (with the same proof):
if $H=H(p,U)$ is a local homoclinic class containing hyperbolic periodic points of indices $\alpha<\beta$,
then for any neighborhood $V$ of $\overline{U}$, the local homoclinic class $H(p,V)$
contains hyperbolic periodic points of any index $i\in \{\alpha,\dots,\beta\}$.
\end{rk}
\medskip

\paragraph{\bf Heterodimensional cycles.}
Robust heterodimensional cycle can be obtained from periodic points with different indices
contained in a same chain-recurrence class.
\begin{thm}\cite{bdk}\label{t.diffind}
For any diffeomorphism $f$ in a dense $G_\delta$ subset of ${\rm Diff}^1(M)$, if $H(p)$ is a homoclinic class containing hyperbolic periodic points of different indices, then $H(p)$ contains a robust heterodimensional cycle.
\end{thm}
\medskip

\subsection{Dominated splitting, hyperbolic times}\label{ss.domination}
Let us consider in this section an invariant set $\Lambda$ with a dominated splitting
$T_\Lambda M=E\oplus F$.
\medskip

\paragraph{\bf Extension to the closure.} The following property is classical.

\begin{lemma}\label{l.Ndomination}
An $N$-dominated splitting $T_\Lambda M=E\oplus F$ over an invariant set $\Lambda$
extends uniquely on the closure of $\Lambda$ as an invariant $N$-dominated splitting.
\end{lemma}
\medskip

\paragraph{\bf Adapted norm.}
Using~\cite{gourmelon},
one can replace the Riemannian norm by another one which is
adapted to the domination: for this norm, there exists $\lambda\in (0,1)$
such that for any point $x\in \Lambda$ and any unitary vectors, $v^E\in E_x$ and $v^F\in F_x$ we have
$$\|Df.v^E\|\leq \lambda \|Df.v^F\|.$$
\medskip

\paragraph{\bf Generalization to forward invariant sets.}
Consider a small neighborhood $U^+$ of $\Lambda$
and consider the set $\Lambda^+$ of points whose forward orbit by $f$ is contained in $\overline{U^+}$.
The bundle $E$ on $\Lambda$ extends uniquely as a forward invariant continuous bundle that is tangent to
a center-stable cone.
This is also defined in the following way: for any $x\in \Lambda^+$ and any vectors
$u\in E_x\setminus \{0\}$ and $v\in T_xM\setminus E_x$, we have
$$ \limsup_{n\to \infty} \frac 1 n \log \|Df^n(x).u\|< \liminf_{n\to \infty} \frac 1 n \log \|Df^n(x).v\|.$$
\medskip

\paragraph{\bf Hyperbolic points.}
Let $\Lambda$ be an invariant compact set and $E$ be a continuous invariant bundle over $\Lambda$.
For any $C>0$ and $\sigma\in (0,1)$, we say that a point $x\in\Lambda$ is
\emph{$(C,\sigma,E)$-hyperbolic for $f$} if for any $n\in\NN$, one has
$$\prod_{i=0}^{n-1}\|Df|_{E(f^{i}(x))}\|\le C\sigma^n.$$
For any $x\in\Lambda$, if $f^n(x)$ is $(C,\sigma,E)$-hyperbolic for $f$,
then $n$ is called a \emph{$(C,\sigma,E)$-hyperbolic time} of $x$ for $f$.

Note that the definition of $(C,\sigma,E)$-hyperbolic points extend to the set $\Lambda^+$
introduced above.
One defines similarly the $(C,\sigma,E)$-hyperbolic times for $f^{-1}$. When there is no ambiguity between $f$
and $f^{-1}$, we just say that a point is a $(C,\sigma,E)$-hyperbolic time.

The next lemma gives the existence
of $(1,\sigma,E)$-hyperbolic points.
\begin{lemma}\label{l.infinite-pliss}
Let $\Lambda$ be an invariant set, $E$ be a continuous invariant bundle on $\Lambda$ and consider $\sigma\in (0,1)$.
Then for any $x\in\Lambda$ with the property
$$\limsup_{n\to\infty}\sum_{i=0}^{n-1}\frac{1}{n}\log\|Df|_{E(f^{i}(x))}\|<\log\sigma,$$
there is a $(1,\sigma,E)$-hyperbolic point $y$ for $f$ in the forward orbit of $x$.
\end{lemma}
\begin{proof}

Let $\varphi\colon \Lambda\to \RR$ be the function defined by
$\varphi(x)=\log\|Df{|E}(x)\|-\log(\sigma)$ so that $\limsup_{n\to +\infty} \frac 1 n \sum_{i=0}^{n-1} \varphi(f^i(x))<0$.
If $\frac{1}{n}\sum_{i=0}^{n-1}\varphi(f^i(x))<0$ for all $n\geq 1$ then $x$ is $(1,\sigma, E)$-hyperbolic. Otherwise, there exists $n_0\geq 1$ such that
\begin{itemize}
\item $\frac{1}{n_0}\sum_{i=0}^{n_0-1}\varphi(f^i(x))\geq 0$.
\item $\frac{1}{n}\sum_{i=0}^{n-1}\varphi(f^i(x))< 0$ for all $n>n_0$.
\end{itemize}
This will imply that $y=f^{n_0}(x)$ is $(1,\sigma,E)$-hyperbolic.
Indeed for any $n\geq 1$,
$$\frac{1}{n}\sum_{i=0}^{n-1}\varphi(f^i(y))=
\frac{1}{n}\left(\sum_{i=0}^{n+n_0-1}\varphi(f^i(x))-\sum_{i=0}^{n_0-1}\varphi(f^i(x))\right).$$
By the above equalities, the first sum of the right hand is less than 0 and the second is not less than 0. So the difference is negative and $y$
is $(1,\sigma,E)$-hyperbolic.
\end{proof}

We are interested by finding $(C,\sigma,E)$-hyperbolic points for $f$ and
$(C,\sigma,F)$-hyperbolic points for $f^{-1}$ that are close.

\begin{lemma}\label{l.hyperbolic-time}
Consider a set $\Lambda$ and a constant $\lambda\in (0,1)$ as above.
Consider also $\sigma,\rho\in (0,1)$ such that $\lambda<\rho\sigma$
and a sequence $(y_k)$ in $\Lambda$ such that one of the two following conditions hold:
$$\limsup_{n\to\infty}\sum_{i=0}^{n-1}\frac{1}{n}\log\|Df^{-1}|_{E(f^{-i}(y_k))}\|<\log(\rho ^{-1}),$$
$$\limsup_{n\to\infty}\sum_{i=0}^{n-1}\frac{1}{n}\log\|Df^{-1}|_{F(f^{-i}(y_k))}\|<\log(\sigma).$$

Then, one of the following properties hold:
\begin{itemize}
\item[--] There exists $C>0$, such that infinitely many points $y_k$ are $(C,\sigma,F)$-hyperbolic for $f^{-1}$.
\item[--] For each $k\geq 0$, there exists a backward iterate $z_k$ of $y_k$ which is $(1,\sigma,F)$-hyperbolic
for $f^{-1}$ and any accumulation point of $(z_k)$ is $(1,\rho,E)$-hyperbolic for $f$.
\end{itemize}
\end{lemma}
\begin{proof}
Note that under the first condition, the domination between $E$ and $F$ gives for each $k$,
$$\limsup_{n\to\infty}\sum_{i=0}^{n-1}\frac{1}{n}\log\|Df^{-1}|_{F(f^{-i}(y_k))}\|<\log(\lambda)+\log(\rho ^{-1})<\log \sigma,$$
so that the second condition holds.
Thus, by Lemma \ref{l.infinite-pliss}, one can get a
$(1,\sigma,F)$-hyperbolic point $f^{-\ell_k}(y_k)$ for $f^{-1}$ in the backward orbit of $y_k$. We call $z_k$ the last one
(corresponding to the smallest $\ell_k\geq 0$).
By taking a subsequence if necessary, there are two cases:
\begin{enumerate}

\item There is $N_0\in\NN$ such that $\ell_k\le N_0$.

\item $\lim_{k\to\infty}\ell_k=\infty$.

\end{enumerate}

For case (1), we take
$$C=\max_{x\in\Lambda,1\le n\le N_0}\left\{\prod_{i=0}^{n-1}\|Df^{-1}|_{F(f^{-i}(x))}\|\right\}.$$
It is clear that any $y_k$ is $(C,\sigma,F)$-hyperbolic and the first case of the lemma occurs.
\medskip

For case (2), we claim that for each $k$ and
for any $m\in[0,\ell_k]\cap\NN$, one has
$$\prod_{i=1}^{m}\|Df^{-1}|_{F(f^i(z_k))}\|\ge\sigma^m.$$
This is proved by induction (the case $m=0$ being clear). Assume that it holds up to some integer
$m\in[0,\ell_k-1]\cap\NN$ and that it is not satisfied for $m+1$. For each $0\leq n \leq m$ we have
$$\prod_{i=n+1}^{m+1}\|Df^{-1}|_{F(f^i(f^{-\ell_k}(y_k)))}\|<\sigma^{m+1-n}. $$
This and the fact that $f^{-\ell_k}(y_k)$ is $(1,\sigma,F)$-hyperbolic show that
$f^{m+1-\ell_k}(y_k)$ is also $(1,\sigma,F)$-hyperbolic and this contradicts the maximality of $-\ell_k$.
Hence the claim holds.
\medskip

Any limit point $z$ of the sequence $(z_k)$ thus satisfies for each $m\geq 0$,
$$\prod_{i=1}^{m}\|Df^{-1}|_{F(f^i(z_k))}\|\ge\sigma^m.$$
The domination between $E$ and $F$ thus implies
$$\prod_{i=0}^{m-1}\|Df|_{E(f^i(z_k))}\|\le\sigma^{-m}\lambda^m\le \rho^m.$$
The point $z$ is thus $(1,\rho,E)$-hyperbolic.
\end{proof}

\medskip

\paragraph{\bf Local manifolds.} A \emph{(locally invariant) plaque family} tangent to $E$ is a continuous map $\D$ from the linear bundle $E$ over $\Lambda$ into $M$ satisfying:
\begin{itemize}
\item[--] For each $x\in \Lambda$ the induced map $\D_x:E_x\to M$ is a $C^1$ embedding that is tangent to $E_x$ at the point $\D_x(0)=x.$
\item[--] The family $(\D_x)_{x\in \Lambda}$ of $C^1$-embedding is continuous.
\item[--] The plaque family is locally forward invariant, i.e., there exists a neighborhood $U$ of the section $0$ in $E$ such that for each $x\in \Lambda$, the image of $\D_x(E_x\cap U)$ by $f$ is contained in $\D_{f(x)}(E_{f(x)}). $
\end{itemize}
This definition extends to locally forward invariant plaque families tangent to $E$ over the set $\Lambda^+$.
In a symmetric way, one can consider plaque families tangent to $F$ over the set $\Lambda^-$,
that are locally backward invariant.

Plaque families tangent to $E$ over an invariant set $\Lambda$ always exist from~\cite{hirsch-pugh-shub},
but the same proof gives locally forward invariant plaque families over the set $\Lambda^+$.
\medskip

The following lemma ensures the existence of \emph{local stable manifold} at $(C,\sigma,E)$-hyperbolic points.
The proof is classical (see for instance~\cite[section 8.2]{abc-mesure}).

\begin{lemma}\label{l.manifold1}
Let $\D$ be a locally forward invariant plaque family tangent to $E$ over $\Lambda^+$. Then, for any constants
$C>0$, $\sigma<\rho<1$, there exists $C',r>0$ such that for any $(C,\sigma,E)$-hyperbolic point $x\in
\Lambda^+$, the ball $B^E(x,r)$ in $\D_x$ is contracted by forward iterations. More precisely, for each $r'\in
(0,r)$ and $n\geq 0$, the image $f^n(B^E(x,r'))$ is contained in $\D_{f^n(x)}$ and has diameter smaller than
$C'\rho^n r'$.
\end{lemma}
\medskip

The stable and unstable manifolds of $E$-hyperbolic and $F$-hyperbolic points that are close intersect:

\begin{lemma}\label{l.manifold-intersect}
Let us consider plaque families $\D^{E},\D^F$ tangent to $E,F$ over the sets $\Lambda^+,\Lambda^-$ respectively.
For any $C>0$ and $\sigma\in (0,1)$, there exists $r>0$ such that
if $y\in \Lambda^-$ is $(C,\sigma,F)$-hyperbolic for $f^{-1}$
and $x\in \Lambda^+$ is $(C,\sigma,E)$-hyperbolic for $f$ and if the distance between $x,y$ is smaller
than $r$, then, the local stable manifold at $x$ and the local unstable manifold at $y$ intersect.
\end{lemma}
\begin{proof}
For points $x\in \Lambda^+$ and $y\in \Lambda^-$ that are close, the plaques $\D^E_x$ and $D^F_y$ intersect
at a point $z$ close to $x,y$. By the previous lemma, this point belongs to
the local stable manifold at $x$ and the local unstable manifold at $y$.
\end{proof}

In the case the set $\Lambda$ has a splitting $T_\Lambda M=E'\oplus E^c\oplus F$,
the following result allows to get intersections between stable manifolds tangent to $E'$
and unstable manifolds tangent to $F$.
The argument is the
same as in the proof of Theorem 13 in~\cite{bonatti-gan-wen} and Proposition 3.7 in~\cite{crovisier-palis-faible}.

\begin{lemma}\label{l.manifold2}
Let us assume that $\Lambda$ has a dominated splitting $T_\Lambda M=E\oplus F=(E'\oplus E^c)\oplus F$
where $E^c$ is one-dimensional
and consider plaque families $\D^{E'},\D^F$ tangent to $E',F$ over the sets $\Lambda^+,\Lambda^-$
respectively.
Assume also that $\Lambda^+$ contains a curve $I$ tangent to $E$ and transverse to $E'$.
For any interior point $x$ of $I$ there exists $r>0$ with the following property.

For any $y\in \Lambda^-$ that is $r$-close to $x$,
the plaque $\D^F_{y}$ intersects the plaque $\D^{E'}_{x'}$
for some point $x'\in I$.
In particular for any $C>0$ and $\sigma\in (0,1)$, if $r$ is small enough,
the Lemma~\ref{l.manifold1} implies that if the point $y$ is $(C,\sigma,F)$-hyperbolic for $f^{-1}$
and if the points of $I$ are $(C,\sigma,E')$-hyperbolic for $f$,
then the local unstable manifold of $y$ and the local stable manifold of $x'$ intersect.
\end{lemma}

\begin{rema}\label{r.manifold}
Under the setting of the previous lemma and in the case $E'=E^s$ is uniformly contracted,
we note that any point in $I$ is $(C,\sigma,E')$-hyperbolic for $f$.
\end{rema}

\medskip

\paragraph{\bf A weak shadowing lemma for measures.}
We now state a result regarding hyperbolic measures whose Oseledets splitting is dominated.
The argument is classical, see~\cite[Proposition 1.4]{Cro08}.
The version we state now is local but the proof is the same.

\begin{prop}\label{hypm}
Let $f$ be a $C^1$ diffeomorphism and $\mu$ be an ergodic measure
which has no zero Lyapunov exponent and whose Oseledets splitting $E^s\oplus E^u$ is dominated. Then $\mu$ is supported on a homoclinic class. Indeed,
for any neighborhood $U$ of the support of $\mu$,
there is a sequence of hyperbolic periodic orbits $O_n$ whose index is equal to $\dim (E^s)$ that are homoclinically related together in $U$ and converge toward the support of $\mu$ in the Hausdorff topology.
\end{prop}

\section{Theorem \ref{main2} implies Theorem \ref{main}}\label{s.conclusion}

In this section we prove Theorem \ref{main} assuming that Theorem \ref{main2} holds. The proof will be based on several results stated in the previous Section \ref{pre}.
\begin{lemma}\label{l.class-domina}
Let $f$ be a diffeomorphism which belongs to a dense $G_\delta$ subset of $\diff^1(M)\setminus \overline{\operatorname{HT}}$
and let $\C$ be a chain-recurrent class of $f$. If
\begin{itemize}
\item[--] $\C$ has a dominated splitting $T_{\C}M=E\oplus F$,

\item[--] $\C$ is the Hausdorff limit of periodic orbits of index $\dim E$,

\item[--] $\C$ does not contains any periodic point whose index is less than $\dim E$,

\item[--] $E$ is not uniformly contracted,

\end{itemize}
then, there exists a dominated splitting $T_{\C}M=E'\oplus E^c\oplus F'$
where $\dim E^c=1$ and $\dim(E')<\dim(E)$ and $\C$ is the Hausdorff limit of periodic orbit of index $\dim(E')$.
\end{lemma}
\begin{proof}
Since $E$ is not uniformly contracted, Theorem \ref{C3options} applies.
By our assumptions, the first case in the conclusion does not occur.

If the second case of Theorem \ref{C3options} occurs, then by Corollary~\ref{c.local-class}
for any $\delta>0$ small, $\C$ is the homoclinic class of a periodic orbit whose maximal Lyapunov exponent
along $E$ belongs to $(-\delta,0)$. Then, $\C$ is the Hausdorff limit of periodic orbits of index $\dim(E)-1$
and by Corollary~\ref{weak}, the class has a dominated splitting
$T_{\C}M=E'\oplus E^c\oplus F$, as required.

We thus assume that $\C$ contains a minimal set $K$ with a partially hyperbolic splitting
$T_{K}M=E^s\oplus E^c\oplus E^u$ satisfying the third case of Theorem \ref{C3options}.
We can take such $K$ with least $\dim E^s$: for any other minimal set $K'\subset \C$ having a dominated splitting
$T_{K'}M=F^s\oplus F^c\oplus F^u$ and such that any invariant probability measure supported on $K'$
has a zero Lyapunov exponent along $F^c$, we have $\dim(E^s)\leq \dim(F^s)$.

Consider the
following family $\F$ of compact invariant chain-transitive subsets of $\C$, say $\La$, such that:
\begin{itemize}
\item[--] $\La$ contains $K.$
\item[--] $T_\La M=E'\oplus E^c\oplus F'$ dominated with $\dim E'=\dim E^s$ and $\dim(E^c)=1$.
\item[--] $\La$ is the Hausdorff limit of periodic orbits $\mO$ with Lyapunov exponent along $E^c$ is arbitrarily weak (i.e., for any $\de>0$ there exists a sequence $\mO_n$ of periodic orbits converging in the Hausdorff topology to $\La$ such that the Lyapunov exponent along $E^c$ belongs to $(-\de,\de).)$
\end{itemize}
Notice that the last condition make sense since every periodic orbit near $\La$ has a splitting $E\oplus
E^c\oplus F.$
Note also that $K$ belongs to $\F$ by Theorem~\ref{t.hausdorff}.
Order the family $\F$ by inclusion, that is $\La_1\le \La_2$ if $\La_1\subset \La_2.$ We will
apply Zorn's Lemma to get a maximal set within $\F.$ Let $\{\La_\al:\al\in\Gamma\}$ be a totally ordered chain and
consider
$$\La_\infty=\overline{\bigcup_{\al\in\Gamma} \La_\al}.$$
It follows that $\La_\infty$ is the Hausdorff limit of periodic orbits whose $(\dim E' +1)^{\rm th}$ Lyapunov
exponent is arbitrarily weak. Thus, by Theorem \ref{wen}, the splitting along these periodic orbits extends to
the closure and so $\La_\infty$ has a dominated splitting $E'\oplus E^c\oplus F'.$ Thus, $\La_\infty$ is an upper
bound set of $\{\La_\al:\al\in\Gamma\}.$ Now, applying Zorn's Lemma we get that there exists a maximal set $\La\in\F$ with respect to the order.

In the case $\La=\C$ we obtain immediately the conclusion of the lemma.
We thus assume that $\La\neq\C$. With the maximality of $\Lambda$,
Proposition~\ref{spread} implies that there exists a measure supported on $\Lambda$ whose Lyapunov exponent
along $E^c$ is not zero. Note also that the bundle $E'$ is uniformly contracted since otherwise
Theorem \ref{C3options} would contradict the minimality of $\dim(E')=\dim(E^s)$ for $K$.
Theorem~\ref{main2} can thus be applied.

By our assumptions, the first property of Theorem~\ref{main2} does not happen.
The second property holds and implies that $\C$ is the Hausdorff limit of periodic orbits
whose $(\dim(E')+1)^\text{th}$ Lyapunov exponent is arbitrarily close to zero.
By Theorem~\ref{wen} and Lemma~\ref{l.Ndomination}, the splitting $T_\La M=E'\oplus E^c\oplus F'$
on $\La$ extends as a dominated splitting on $\C$.
By Lemma~\ref{l.index}, $\C$ is the Hausdorff limit of periodic orbit of index $\dim(E')$,
ending the proof.
\end{proof}

Let $\C$ be a chain-recurrent class of $f$. Assume first that $\C$ is aperiodic. By Theorem~\ref{t.hausdorff} it
is the Hausdorff limit of a sequence of periodic orbits. By Theorem~\ref{wen}, we get a dominated splitting
$T_\C M=E\oplus F$ such that $\dim(E)$ coincide with the index of the periodic orbits. Since $\C$ is aperiodic,
it is not a hyperbolic set, hence either $E$ is not uniformly contracted or $F$ is not uniformly expanded. We
will assume for instance that the first case holds. One can hence apply Lemma~\ref{l.class-domina} above. This
gives another dominated splitting $T_\C M=E'\oplus E^c\oplus F'$ such that $\C$ is the Hausdorff limit of
periodic orbits of index $\dim(E')<\dim(E)$. If $E'$ is not uniformly contracted, the lemma applies again.
Applying inductively the lemma, one can thus assume that $\La$ has a splitting $T_\C M=E^s\oplus E^c\oplus
E^{cu}$ where $E^s$ is uniformly contracted and $E^{c}$ has dimension $1$ and is not uniformly contracted.
Since $\C$ is aperiodic, the conclusion of Theorem~\ref{main2} does not occur. Consequently, for any invariant
probability measure supported on $\C$, the Lyapunov exponent along $E^c$ is equal to zero. With the domination,
this implies that $E^{cu}$ is uniformly expanded, as in Theorem~\ref{main}.
\medskip

We assume now that $\C$ is a homoclinic class $H(p)$.
By Theorems~\ref{indices} and~\ref{wen}, there exists a dominated splitting
$$T_{H(p)}M=F^{cs}\oplus F^c_1\oplus\dots\oplus F^c_s\oplus F^{cu}$$
such that each bundle $F^{c}_i$ is one-dimensional,
$\dim(F^{cs})$ is equal to the minimal index of periodic orbits of $H(p)$
and $\dim(M)-\dim(F^{cu})$ is equal to the maximal index of periodic orbits of $H(p)$.
Moreover for each $i$, there exist periodic orbits in $H(p)$ whose Lyapunov exponent
along $F^c_i$ is arbitrarily close to zero.

If the bundle $F^{cs}$ is not uniformly contracted, then the Lemma~\ref{l.class-domina} applies.
As before, by an induction we get a dominated splitting
$T_{H(p)} M=E^s\oplus E^c\oplus  E^{cu}$ such that $\dim(E^s)<\dim(F^{cs})$,
the bundle $E^c$ is one-dimensional and not uniformly contracted and $E^s$ is uniformly contracted.
Since $H(p)$ contains hyperbolic periodic points, there exists an invariant measure supported on
$H(p)$ whose Lyapunov exponent along $E^c$ is not zero; hence Theorem~\ref{main2} can be used.
The first case of the conclusion does not happen since $\dim(E^s)$ is smaller than the minimal
index of periodic points in $H(p)$. The second case is thus satisfies:
we get that the minimal index $\dim(F^{cs})$ coincides with $\dim(E^s)+1$ so that $F^{cs}=E^s\oplus E^c$.
Moreover there exists periodic points whose Lyapunov exponent along $E^c$ is arbitrarily close to zero.

The same holds for the bundle $F^{cu}$, giving the conclusion of~Theorem~\ref{main}.

\section{Central models and some consequences}\label{central}

When dealing with sets having a dominated splitting where one bundle is one dimensional, the first author
developed a tool, known as ``Central Models'', that proved to be very useful in this context (see
\cite{crovisier-palis-faible}, \cite{Cro08}). We will recall the main classification result regarding these
central models.

\subsection{Definition of central models}
Let $K$ be a compact invariant set with a dominated splitting of the form $T_K M=E_1\oplus E^c\oplus E_3$ where $E^c$ is one dimensional. Consider a locally invariant \emph{plaque family} $\D^c$ tangent to $E^c$.
This always exists: using a result of~\cite{hirsch-pugh-shub} recalled in Section~\ref{ss.domination},
one can consider center-stable and center-unstable plaque families $\D^{cs},\D^{cu}$ tangent to the bundles
$E_1\oplus E^c$ and $E^c\oplus E_3$ respectively and choose $\D^c$ such that for each $x\in K$, the image of
$\D^c_x$ is contained in the intersection of the images of $\D^{cs}_x$ and $\D^{cu}_x$.

Plaque families tangent to $E^c$ is a way to introduce central manifolds for points $x\in K$.
When such a family $\D^c$ is fixed and $\eta>0$ small is given, one defines the \emph{central manifold
of size $\eta$} at $x\in K$ as
$$W^c_\eta(x)=\D^c_x(-\eta,\eta).$$
\medskip

The above properties of the central plaque family allow to lift the dynamics of $f$ as a fibered dynamics $\hat
f$ on the bundle $E^c$ that is locally defined in a neighborhood of the (invariant) zero section of $E^c.$ (See
\cite{crovisier-palis-faible} and \cite{Cro08} for details.) First, one notices that the action induced by $f$
on the unitary bundle associated to $E^c$ is the union of one or two chain-recurrence classes. We denote $\hat
K$ one of these classes: this can be viewed as a collection of half central plaques of $\D^c$. These half
plaques can be parameterized, giving a projection map $\pi\colon \hat K\times [0,+\infty)\to M$ such that
$\pi(\{\hat x\}\times [0,+\infty))\subset \D^c_x$ and $\pi(\hat x,0)=x$ for each $\hat x\in \hat K$ above a
point $x\in K$. Since the plaque family $\D^c$ is locally invariant, the map $f$ can be lifted as a map $\hat f$
defined on a neighborhood of $\hat K\times \{0\}$ which is fiber preserving and locally invertible. The set
$\hat K\times [0,+\infty)$ endowed with the map $\hat f$ is called a \emph{central model} of $(K,f)$ for the
bundle $E^c$.

Sometimes it is useful to consider central models associated to smaller plaques. When $\eta>0$ is given we say
that the central model is \emph{associated to the plaques $W^c_\eta$} if for each $\hat x\in \hat K$ the
projection $\pi(\{\hat x\}\times [0,+\infty))$ is contained in $W^c_\eta(x)$ where $x=\pi(\hat x,0)$.
\medskip

Two cases occur when $K$ is chain-transitive:
\begin{itemize}
\item[--] Either the action induced by $f$ on the unitary bundle associated to $E^c$ has only one
chain-recurrence class. Equivalently, there is no continuous orientation of $E^c$ preserved by the action of $Df$.
In this case $\hat K$ is a two-folds covering of $K$. Moreover if $\hat x^-,\hat x^+\in \hat K$ are the two lifts of $x\in K$, then the union $\pi(\{\hat x^-\}\times [0,+\infty))\cup \pi(\{\hat x^+\}\times [0,+\infty))$
contains a uniform neighborhood of $x$ in $\D^c_x$.
\item[--] Or the unitary bundle associated to $E^c$ is the union of two chain-recurrence classes
$\hat K^-,\hat K^+$. In this case, $\hat K^-, \hat K^+$ are copies of $K$.
Considering any two central models $\pi^\pm\colon \hat K^\pm\times [0,+\infty)\to M$, then for any lifts
$\hat x^\pm\in \hat K^\pm$ of a point $x\in K$, the union $\pi(\{\hat x^-\}\times [0,+\infty))\cup \pi(\{\hat x^+\}\times [0,+\infty))$ contains a uniform neighborhood of $x$ in $\D^c_x$.
\end{itemize}

\subsection{Classification of central dynamics}
Let us consider a central model $\hat f\colon \hat K\times [0,1]\to \hat K\times [0,+\infty)$.
We introduce some definitions used to describe its dynamics.

\begin{defi}
\noindent a) When $\hat Z$ is a subset of $\hat K$, we define the \emph{chain-unstable set of $\hat Z$}
for the dynamics of the central model,
$\hat W^{ch-u}(\hat Z)$ as the set of points $y\in \hat Z\times [0,\infty)$ such that
for any $\varepsilon>0$, there exists a sequence $\{x_0,\dots,x_n\}$ in $\hat Z$
and a sequence $\{y_0,\dots,y_n\}$ in $\hat Z\times [0,1]$ satisfying:
\begin{itemize}
\item[--] for each $k\in \{0,\dots,n\}$ we have $y_k\in \{x_k\}\times [0,1]$,
\item[--] $y_0=(x_0,0)$ and $y_n=y$,
\item[--] for each $k\in \{1,\dots,n\}$, the points $y_k$ and $\hat f(y_{k-1})$ are $\varepsilon$-close.
\end{itemize}
We define similarly the \emph{chain-stable set of $\hat Z$}. Note that we did not assume that $\hat Z$ is compact nor invariant.
\smallskip

\noindent
b) A point $\hat x\in \hat K$ has a \emph{chain-recurrent central segment} if there exists $a>0$
such that $\{\hat x\}\times [0,a]$ is contained in the chain-stable and chain-unstable sets of $\hat K\times \{0\}$.
\smallskip

\noindent c) We say that the dynamics of $\hat f$ is \emph{thin trapped} if there exists arbitrarily small open neighborhoods $U$
of $\hat K\times \{0\}$ such that $\hat f(\overline U)\subset U$ and for each $\hat x\in \hat K$
the intersection $U\cap (\{\hat x\}\times [0,+\infty))$ is an interval.
\end{defi}
\medskip

The next classification results are the basis of this theory. They restate~\cite[section 2]{crovisier-palis-faible}
and~\cite[proposition 2.2]{Cro08}
\begin{thm}\label{cm}
Let $(\hat K\times [0,+\infty),\hat f)$ be a central model and assume that $K$ is chain-transitive.
Then, the following properties hold.
\begin{enumerate}
\item The dynamics of $\hat f$ is not thin trapped if and only if
the chain-unstable set of $\hat K$ contains a non-trivial interval $\{\hat y\}\times [0,a]$, $a>0$.
\item The dynamics of $\hat f$ and $\hat f^{-1}$ are not thin trapped
if and only if there is a chain-recurrent central segment.
\item If the dynamics of $\hat f$ is thin trapped and the dynamics of $f^{-1}$ is not
thin trapped then there is a neighborhood of $\hat K\times \{0\}$
contained in the chain-stable set of $\hat K\times \{0\}$.
\end{enumerate}
\end{thm}
\begin{proof}
Let us explain the first item:
\begin{itemize}
\item[--] In the case the chain-unstable set of $\hat K\times \{0\}$ is reduced to $\hat K\times \{0\}$, we apply~\cite[lemma 2.7]{crovisier-palis-faible} (the $\varepsilon$-chain-unstable sets of $\hat K\times \{0\}$
are arbitrarily small trapped strips): the dynamics is thin trapped.
\item[--] In the case the chain-unstable set of $\hat K\times \{0\}$ contains a point $(\hat y,a)$ with $a>0$, we note that it also contains the interval $\{\hat y\}\times [0,a]$. Clearly, the dynamics can not be thin trapped.
\end{itemize}

The second item is~\cite[proposition 2.5]{crovisier-palis-faible}.
For the third item: if the dynamics of $\hat f$ thin trapped and the dynamics of $f^{-1}$ is not thin trapped, then there is no chain-recurrent central segment
and by the first item the chain-stable set is of $\hat K\times \{0\}$ is not trivial.
By~\cite[lemma 2.8]{crovisier-palis-faible}, this implies that
there is a neighborhood of $\hat K\times \{0\}$
contained in the chain-stable set of $\hat K\times \{0\}$.
\end{proof}
\bigskip

We now come back to the manifold $M$ and discuss the central dynamics of $K$.
\begin{defi}
\noindent a) $K$ has a \emph{chain-recurrent central segment} if in some central model of $(K,f)$, there
exists a chain-recurrent central segment. More precisely, we say that a point $x\in K$ has a chain-recurrent
central segment with respect to a central model of $(K,f)$ if there exists a chain-recurrent central segment
$\{\hat x\}\times [0,a]$, $a>0$, in this central model where $\hat x$ lifts $x$.
\smallskip

\noindent b) The central dynamics is \emph{thin trapped} if there exist one or two central models
$(\hat K\times [0,+\infty), \hat f)$ or $(\hat K^\pm\times [0,+\infty), \hat f^\pm)$ that are thin trapped
and such that $\hat K$ or
$(\hat K^-\cup \hat K^+)$ contain the whole unitary bundle associated to $E^c$.
\end{defi}

From the classification result, if $K$ is chain-transitive and has no central model with non-trivial
chain-unstable set, then the central dynamics of $K$ is thin trapped. Moreover, this does not depend on the choice of
the central models:

\begin{lemma}\label{l.model-uniqueness}
Assume that $K$ is chain-transitive and that its central dynamics is thin trapped.
Then, for any central model $(\hat K\times [0,+\infty),\hat f)$ of $K$,
the dynamics is thin trapped.
\end{lemma}
\begin{proof}
This is contained in the proof of lemma 2.5 in~\cite{Cro08}.
\end{proof}

\subsection{Dynamics with chain-recurrent central segment}
For next result of this section we need a preliminary lemma.

\begin{lemma}\label{basic}
Let $f$ be a diffeomorphism, $\La$ be a compact invariant chain-transitive set with a dominated splitting
$E\oplus E^c\oplus F$ where $E^c$ is one dimensional and $(W^c(x))$ be a locally invariant plaque family tangent to $E^c$ over $\Lambda$. There is $\theta_0>0$ such that for any small neighborhood $U_\Lambda$ of $\Lambda$, any $\eta>0$ small, any arc $I$ satisfying:
\begin{itemize}
\item[--] there exists $x\in \Lambda$ satisfying $f^k(I)\subset W^c_\eta(f^k(x))$ for all $k\ge 0$;
\item[--] $I\subset W^{ch-s}(\La,U_\Lambda)$;
\item[--] there exists a sequence of periodic points $p_n$ whose orbit $\mO(p_n)$ is contained in $U_\Lambda$,
whose central exponent belongs to $(-\theta_0,\theta_0)$ and such that $p_n$ converge to an interior point $y$ of $I$;
\end{itemize}
then, for some $n$ large enough and any neighborhood $U'_\Lambda$ of $\overline{U_\Lambda}$ we have that
$$W^u_\eta(\mO(p_n))\cap W^{ch-s}(\La,U'_\Lambda)\neq\emptyset.$$
\end{lemma}
\begin{proof}
Let us assume that the Riemannian norm is adapted to the domination $(E\oplus E^c)\oplus F$ and that
$\lambda\in(0,1)$ is a constant as in Section~\ref{ss.domination}.

Consider $\varepsilon>0$ small, then if $\eta>0$ is small enough,
and since all the curves $f^{k}(I)$, $k\geq 0$ have length smaller than $\eta$ we have for each $z\in I$
$$\|Df^{k}_{|E^c}(z)\|\leq \frac{\operatorname{length}(f^{k}(I))}{\operatorname{length}(I)}e^{\varepsilon.k}.$$
With the domination, this implies that there exist $C_E>0$ and $\sigma_E\in (0,1)$ such that any point $z\in I$ is
$(C_E,\sigma_E,E)$-hyperbolic for $f$.

We choose $\rho,\sigma_F \in (0,1)$ such that $\lambda<\rho\sigma_F$ and set $\theta_0=\log( \rho^{-1})$.
We then apply the Lemma~\ref{l.hyperbolic-time} to the sequence $(p_n)$ and discuss the two cases of conclusion.

In the first case, taking a subsequence, the points $(p_n)$ are $(C_F,\sigma_F,F)$-hyperbolic for $f^{-1}$ and some constant $C_F>0$. We then set $C=\sup(C_E,C_F)$ and $\sigma=\sup(\sigma_E,\sigma_F,\rho)$.
From Lemma~\ref{l.manifold2}, for $n$ large there exists an intersection between the local unstable manifold
of $p_n$ in the plaque tangent to $F$ at $p_n$ and the local stable manifold of some point $z_n\in I$
close to $y$ in the plaque tangent to $E$ at $y_n$. The result follows in this case.

In the second case, there exist points $q_n$ in the orbit of $p_n$ that are $(1,\sigma_F,F)$-hyperbolic for $f^{-1}$
and which converge to some point $z$ such that:
\begin{itemize}
\item[--] $z$ belongs to the maximal invariant set in $\overline{U_\Lambda}$,
\item[--] $z$ is $(1,\rho, E^c)$-hyperbolic for $f$,
\item[--] $z$ belongs to $W^{ch-s}(\Lambda,\overline{U_\Lambda})$.
\end{itemize}
By Lemma~\ref{l.manifold-intersect}, for $n$ large the points $z$ and $p_n$ are close, hence the local
stable manifold at $z$ and the local unstable manifold at $p_n$ intersect.
Since $z$ belongs to $W^{ch-s}(\Lambda,\overline{U_\Lambda})$, one deduces that $p_n$
belongs to $W^{ch-s}(\Lambda,U'_\Lambda)$.
\end{proof}

Applying the Lemma~\ref{basic} to $f$ and $f^{-1}$, one gets the following corollary.
\begin{coro}\label{c.central-segment}
Let $f$ be a diffeomorphism, $\La$ be a compact invariant chain-transitive set with a dominated splitting $E\oplus E^c\oplus F$ where $E^c$ is one dimensional and $\D$ be a plaque family tangent to $E^c$.

There is $\theta_0>0$ such that for any small neighborhood $U_\Lambda$ of $\Lambda$ and any $\eta>0$ small
and considering
\begin{itemize}
\item[--] a chain-recurrent central segment $I$ of $K$ associated to some central model and to the plaques
$W^c_\eta$,
\item[--] a sequence $(p_n)$ of periodic points, which converge to an interior point $y$ of $I$,
whose orbit is contained in $U_\Lambda$, and whose central Lyapunov exponent belongs to $(-\theta_0,\theta_0)$,
\end{itemize}
then, for any neighborhood $U'_\Lambda$ of $\overline{U_\Lambda}$
and some $n$ large, the point $p_n$ belongs to the chain-stable and chain-unstable sets
$W^{ch-s}(\Lambda, U'_\Lambda)$ and $W^{ch-u}(\Lambda, U'_\Lambda)$.
\end{coro}

\subsection{Thin trapped central dynamics}
When $K$ is hyperbolic, it is contained in a local homoclinic class.
The following result is analogous in the case the central dynamics is thin trapped.

\begin{prop}\label{NS}
Let $f$ be a diffeomorphism and $\Lambda$ be a compact invariant chain-transitive set. Assume that:
\begin{itemize}
\item[--] There exists a compact invariant set $K\subset \Lambda$ having a partial hyperbolic structure $T_KM=E^s\oplus E^c\oplus E^u$ with $E^c$ one dimensional.
\item[--] The central dynamics of $K$ is thin trapped.
\item[--] There are a neighborhood $U_\Lambda$ of $\Lambda$, some locally invariant plaque family
$\D^c$ tangent to $E^c$ over $K$ and some plaque $\D^c_x$ contained in the chain-stable set $W^{ch-s}_{U_\Lambda}(\La)$ of $\La$ in $U_\Lambda$.
\end{itemize}
Then, for any neighborhoods $U'_\Lambda,U_K$ of $\overline{U_\Lambda}$ and $K$
there exists a periodic point $p$ such that:
\begin{itemize}
\item[--] The whole orbit of $p$ is contained in $U_K$.
\item[--] $p$ is contained both in the local chain-stable and chain-unstable sets of $\Lambda$ in $U_\Lambda'$.
\end{itemize}
\end{prop}

\begin{proof}
Note first that one can assume that \emph{any} plaque $\D^c_y$
is contained in the local chain-stable set $W^{ch-s}_{U_\Lambda}(\La)$.
Indeed if $(z_0,\dots,z_n)$ is a $\varepsilon>0$ pseudo-orbit between $z_0=y$ and $z_n=x$
with $\varepsilon>0$ small enough, then, since the plaque family $\D$ is trapped,
$f(\D_{z_k})$ is contained in a small neighborhood of $\D_{z_{k+1}}$ for each $k$.
This implies that any point of $\D_y$ can be joint to $\D_x$ by pseudo-orbit.

Note also that we can always replace $K$ by a minimal subset, so that one can assume that $K$ is chain-transitive.
Let $U$ be a small neighborhood of $K$, so that the partially hyperbolic structure $E^s\oplus E^c\oplus E^u$
on $K$ extends to the maximal invariant set in $U$. This allows to introduce some locally invariant plaque
families $(W^{cs}(x))$ and $(W^c(x))$ tangent to $E^s\oplus E^c$ and $E^c$ over this maximal invariant set.
We may assume that $\overline{W^c(x)}\subset W^{cs}(x)$ for each $x$.

By assumption and Lemma~\ref{l.model-uniqueness}, for any central model $(\hat K\times [0,+\infty),\hat f)$
over $K$ associated to the plaque family $(W^c(x))$, the dynamics is thin trapped.
We claim that there exists a continuous map $\delta\colon \hat K\to (0,+\infty)$ such that
for each $\hat x\in \hat K$, we have
\begin{equation}\label{e.trapped}
\hat f(\{\hat x\}\times [0,\delta(\hat x)])\subset \{\hat f(\hat x)\}\times [0,\delta(\hat f(\hat x))).
\end{equation}
Indeed, since the dynamics is thin trapped, there exists an open neighborhood $V$ of $\hat K\times \{0\}$ in the
central model such that $\hat f(\overline V)\subset V$. Let $S$ be the union of intervals $\{\hat x\}\times
[0,a]$ contained in $V$. Since $V$ is open, this is an open set of the form $\{\{\hat x\}\times [0,a(\hat
x))\}$, where $a\colon \hat K\to (0,+\infty)$ is lower semi-continuous. The image of $\overline S$ by $\hat f$
is contained in $U$. Since it is a union of segments $\{\hat x\}\times [0,a]$, it is contained in $S$. One deduces that
$\hat f(\overline S)$ is a compact set of the form $\{\{\hat x\}\times [0,b(\hat x)]\}$, where $b\colon \hat K\to (0,+\infty)$ is
upper semi-continuous. Hence, there exists continuous map $\delta\colon \hat K\to (0,+\infty)$ such that
$b(x)<\delta(x)<a(x)$ at every point (such a map exists locally and can be obtained globally by a partition of
the unity). The claim follows.

The property~\eqref{e.trapped} above shows that one can reduce the plaques $W^c(x)$ so that for each $x\in K$
we have
$$f(\overline{W^c(x)})\subset W^c(f(x)).$$
Note that this property extends to any point in the maximal invariant set in a small neighborhood $U'$ of $K$.
We will use the following properties.

\begin{enumerate}
\item
By~\cite[Lemma 3.11]{crovisier-palis-faible}, there exists $\varepsilon>0$, such that for any periodic orbit
$\cO\subset U'$, any point in the $\varepsilon$-neighborhood of $W^c(q)$ inside the plaque $W^{cs}(q)$
for $q\in \cO$, belongs to a periodic point $p\in W^c(q)$. Having chosen $U'$ and the plaques $W^c(x)$ small
enough, such a periodic point $p$ is arbitrarily close to some point $z\in K$.

\item Since the central dynamics of $K$ is thin trapped, for any $z\in K$, any point inside the plaque $\D^c_z$
belongs to $K^+$. When $p$ is close enough to $z\in K$ one can apply Lemma~\ref{l.manifold2} and
Remark~\ref{r.manifold}: the strong unstable
manifold of $p$ meets the strong stable manifold of a point $z'\in \D^c_z$ close to $z$.
Since by assumption $z'$ belongs to
the chain-stable set of $\Lambda$ inside $U_\Lambda$, the periodic point $p$ belongs to the chain-stable set of
$\Lambda$ inside $U'_\Lambda$.
\end{enumerate}
\medskip

By~\cite[Lemma 2.9]{Cro08} any partially hyperbolic set $K$ whose central bundle is one-dimensional and thin trapped
satisfies the shadowing lemma: for any $\delta>0$, there exists $\varepsilon>0$ such that any
$\varepsilon$-pseudo-orbit in $K$ is $\delta$-shadowed by an orbit in $M$. This implies that there exists a periodic
orbit $\cO$ contained in an arbitrarily small neighborhood of $K$. In particular, there exists $y\in K$ such
that the local strong unstable manifold of $y$ meets the $\varepsilon$-neighborhood of $W^c(q)$ inside the
plaque $W^{cs}(q)$ for some $q\in \cO$.
By the first property above, this implies that there is a periodic point $p\in W^c(q)$
which belongs to the local chain-unstable set of $\Lambda$.
By the second property, $p$ belongs to the local chain-stable set of $\Lambda$
inside $U'_\Lambda$, as required.
\medskip

Note that when the bundle $E^s$ (or $E^u$) is trivial
the local strong stable (or unstable) manifold of any point $x$ is reduced to $x$
but the proof is unchanged.
\end{proof}

\begin{rema}\label{r.related}
We can have the following stronger statement, which does not suppose the uniform expansion along $E^u$.
\smallskip

\noindent
\emph{Let $\Lambda$ be a compact invariant chain-transitive set. Assume that:
\begin{itemize}
\item[--] There is a compact invariant set $K\subset \Lambda$ and a dominated splitting $T_KM=E^s\oplus E^c\oplus F$ with $E^c$ one dimensional and $E^s$ uniformly contracted.
\item[--] The central dynamics of $K$ is thin trapped.
\item[--] There exists a $(C,\sigma,F)$-hyperbolic point for $f^{-1}$ in $K$, for some constants $C,\sigma$.
\item[--] There are a neighborhood $U_\Lambda$ of $\Lambda$, some locally invariant plaque family
$\D^c$ tangent to $E^c$ over $K$ and some plaque $\D^c_x$ contained in the chain-stable set $W^{ch-s}_{U_\Lambda}(\La)$ of $\La$ in $U_\Lambda$.
\end{itemize}
Then, for any neighborhoods $U'_\Lambda,U_K$ of $\overline{U_\Lambda}$ and $K$
there exists a periodic point $p$ such that:
\begin{itemize}
\item[--] The whole orbit of $p$ is contained in $U_K$.
\item[--] $p$ is contained both in the local chain-stable and chain-unstable sets of $\Lambda$ in $U_\Lambda'$.
\end{itemize}}
\smallskip

We will use this result only in the alternative proof of Proposition~\ref{p.aperiodic},
at section~\ref{ss.aperiodic}. For this reason, we only sketch the proof.
Let us assume that the Riemannian norm is adapted to the domination
$(E^s\oplus E^c)\oplus F$ and consider $\lambda\in (0,1)$ a constant as in Section~\ref{ss.domination}. As in the proof of Proposition~\ref{NS},
any plaque of the family $\D^c$ is contained in $W^{ch-s}_{U_\Lambda}(\La)$.
\medskip

If $K$ supports an ergodic measure which is hyperbolic (i.e. all its Lyapunov exponents are non-zero),
then the conclusion follows. Indeed, by Proposition~\ref{hypm}, there exists a sequence of hyperbolic periodic orbits that converges towards a subset of $K$ and that are homoclinically related together in a small neighborhood of $K$.

If $E^c$ is uniformly contracted, since there exists a $(C,\sigma,F)$-hyperbolic point
for $f^{-1}$, there exists an ergodic measure $\mu_0$ such that the integral of the function $\log \|Df_{|F}^{-1}\|$ is smaller than $\sigma$.
Such a measure is hyperbolic and we are done.

If $K$ supports an ergodic measure $\mu$ whose central Lyapunov exponent is non-zero
and larger than $\log \lambda$,
then by the domination, all the exponents of $\mu$ along $F$ are positive, the measure is hyperbolic and we are done also.
\medskip

If for some invariant compact set $K'\subset K$, any invariant probability measure
on $K'$ has a central Lyapunov exponent equal to zero, then
Proposition~\ref{NS} can be applied and the conclusion follows.
One deduces that $E^c$ is not uniformly contracted on $K$ and
for any invariant compact set $K'\subset K$,
there exists a measure whose central Lyapunov exponent is smaller than $\log \lambda$.

Liao's selecting argument thus applies (see~\cite[lemma
3.8]{wen-conjecture}): again, there exists a sequence of hyperbolic periodic orbits that converges towards a
subset of $K$ and that are homoclinically related together in a small neighborhood of $K$, concluding the proof.
\end{rema}

\section{Proof of Theorem \ref{main2}}\label{s.main2}
In this section we reduce the proof of Theorem \ref{main2} to some technical propositions that will be proved in
other sections. Let $\La$ be as in the hypothesis of it with a dominated splitting $E^s\oplus E^c\oplus F$ with $E^c$ one dimensional and $ E^s$ uniformly contracted. We consider a small neighborhood $U_0$ of $\Lambda$, a small constant
$\theta$ and we have to show that the Theorem~\ref{main2} is satisfied by $U_0$ and $\theta$. We also fix a locally invariant family of central plaques $\D^c$
over the set $\La$. All the central manifolds $W^c_\eta(x)$ and all central models that we will consider are
taken from this plaque family.

\subsection{Generic assumption}\label{ss.generic}
Consider $U,V$ open sets of $M$ and $\theta>0$. We define $\cO(U,V,\theta)$
as the set of diffeomorphisms of $M$
having a hyperbolic periodic orbit $O$ contained in $U$ which meets $V$ and whose
$(\dim(E^s)+1)^\text{th}$ Lyapunov exponent belongs to $(-\theta,\theta)$.

Let $\cB_0$ be a countable basis of open sets of $M$ and $\cB$ the open sets of $M$
that are finite union of elements of $\cB_0$. It has the following properties: it is countable and
for each compact set $K\subset M$ and each open set $V\subset M$ containing $K$, there exists $U\in \cB$
such that $K\subset U\subset V$.

We then define the dense $G_\delta$ set:
$$\cG=\bigcap_{(U,V,\theta)\in \cB\times \cB\times \QQ^+} \left(\cO(U,V,\theta)\cup (\diff^1(M)\setminus \overline{\cO(U,V,\theta)})\right).$$
\medskip

The dense $G_\delta$ subset of $\diff^1(M)$ given by Theorem~\ref{main2}
is the set of diffeomorphisms whose periodic orbits are hyperbolic and that belong to the intersection of
the dense $G_\delta$ sets given by Theorem~\ref{t.hausdorff},
Corollary~\ref{c.local-class}, Corollary~\ref{c.central-segment} and of $\cG$ defined above.

\subsection{Measures supported on $\Lambda$}
For an ergodic measure $\mu$ supported on $\La$ we denote by $L^c(\mu)$ the Lyapunov exponent of $\mu$ along $E^c.$
Note that this also defines the Lyapunov exponent $L^c(O)$  along $E^c$ of a periodic orbit $O\subset U_0$.

We know that there exists an ergodic measure $\mu$ such that $L^c(\mu)\neq 0.$ If this happen to be positive, i.e., $L^c(\mu)>0$ then by the domination on $F$ and since $E^s$ is uniformly contracted we have that $\mu$ is a hyperbolic measure whose Oseledets splitting $E^s_\mu\oplus E^u_\mu$ coincides a.e. with $E^s\oplus (E^c\oplus F).$ Now,  Proposition \ref{hypm} yields the first option of Theorem \ref{main2}.

Thus, assume from now on that for any ergodic measure $\mu$ in $\La$ we have that $L^c(\mu)\le 0.$ The next lemma says that if there are ergodic measures with negative central Lyapunov exponent arbitrarily close to zero we can conclude.

\begin{lemma}
Assume that for every $\epsilon>0$ there exists an ergodic measure $\mu$ in $\La$ such that $L^c(\mu)\in(-\epsilon,0).$ Then, the second option of Theorem \ref{main2} holds.
\end{lemma}
\begin{proof}
Let $\mu$ be an ergodic measure supported on $\Lambda$ with $L^c(\mu)\in(-\epsilon,0).$
By the domination between $E$ and $F$, if $\epsilon$ is sufficiently small, we have that any Lyapunov exponent of $\mu$ along $F$ is positive. On the other hand, since $E^s$ is uniformly contracted  we have that any Lyapunov exponent of $\mu$ along $E^s$ is negative. Therefore, $\mu$ is a hyperbolic measure whose Oseledets splitting $E^s_\mu\oplus E^u_\mu$ coincides a.e. with $(E^s\oplus E^c)\oplus F.$ Proposition \ref{hypm} yields now the second option of Theorem \ref{main2}.
\end{proof}

Continuing with the proof we have to handle the following situation that we shall assume from now on:
\begin{itemize}
\item[(*)] \textit{There exists $\ep_0$ such that any ergodic measure $\mu$ in $\La$ satisfies
$$L^c(\mu)\cap (-\ep_0,0)=\emptyset.$$
Moreover there exists an ergodic measure $\mu_0$ such that $L^c(\mu_0)\le -\ep_0.$}
\end{itemize}

\subsection{$0$-CLE sets $K$.} A \emph{zero central Lyapunov exponent set} (or shortly a \emph{$0$-CLE set})
is a compact invariant chain-transitive set $K\subset \Lambda$ such that $L^c(\mu)=0$ for any
ergodic measure supported on $K$.

The Theorem~\ref{C3options} restates as:

\begin{lemma}
Assume that (*) holds. Then one of the following hold:
\begin{enumerate}
\item The first or the second option of Theorem \ref{main2} is true.
\item There exists a $0$-CLE set $K\subset \La$.
\end{enumerate}
\end{lemma}
\begin{proof}
Let us apply Theorem~\ref{C3options} to the dominated splitting $E\oplus F=(E^s\oplus E^c)\oplus F$ on $\Lambda$.
The two first cases of Theorem~\ref{C3options} give the two options of Theorem~\ref{main2}.
The third case gives a $0$-CLE set $K\subset \Lambda$.
\end{proof}

Therefore, in order to prove our theorem, by the above lemma we will restrict ourselves to the following case:

\begin{itemize}
\item[(**)] \textit{There exists a $0$-CLE set $K\subset \La$.}
\end{itemize}

\subsection{The set $X$ of hyperbolic points} As explained in Section~\ref{ss.domination},
one can change the Riemannian norm and assume that there exists $\lambda\in (0,1)$
such that for any point $x\in \Lambda$ and any unitary vectors, $v^{cs}\in E^{s}_x\oplus E^c_x$ and $v^F\in F_x$ we have
$$\|Df.v^{cs}\|\leq \lambda \|Df.v^F\|.$$

We introduce $\epsilon\in (0,\epsilon_0)$ where $\epsilon_0$ is defined by (*)
and we assume that $e^{-\epsilon}>\lambda^{1/2}$.
Let us define the set $X$ of points that are $(1,e^{-\epsilon},E^c)$-hyperbolic for $f$, that is:
$$X=\{x\in\La: \|Df^n_{/E^c_x}\|\le e^{-n\ep}\,\,\,\forall\,\,\,n\ge 0\}.$$
Notice that $X$ is compact. The next lemma and (*) show that it is non-empty.
\begin{lemma}\label{l.0CLE}
For any invariant ergodic measure $\mu$ such that $L^c(\mu)<0$
the sets $X$ and $\operatorname{supp}(\mu)$ intersect.
\end{lemma}
\begin{proof}
The continuous function $\varphi\colon \operatorname{supp}(\mu)\to \RR$ defined by
$\varphi(x)=\log\|Df_{|E^c}(x)\|$ satisfies $L^c(\mu)=\int \varphi d\mu$
and by (*) it holds $L^c(\mu)< -\varepsilon$.
By Birkhoff's theorem and Lemma~\ref{l.infinite-pliss},
the orbit of a.e. point $x$ meets $X$.
\end{proof}

In particular any compact invariant chain-transitive set $K\subset \Lambda$ that is disjoint from $X$ is a $0$-CLE set.

\subsection{The chain-unstable case}
We will now consider two cases depending if there exist a $0$-CLE set $K$ and a point $y\in \Lambda$ with $\alpha(y)\subset K$
such that the chain-unstable set $\hat W^{ch-u}(\hat K\cup \orb^-(\hat y))$ contains a non-trivial segment $\{\hat y\}\times [0,a]$, $a>0$, in some central model.
When such $K$ and $y$ exist we apply
the following proposition that will be proved in Section~\ref{s.connecting}.

\begin{prop}\label{p.unstable}
Assume that there is a $0$-CLE set $K\subset\Lambda$ and $y\in \Lambda$
such that $\alpha(y)\subset K$ and
for some central model $\hat W^{ch-u}(\hat K\cup \orb^-(\hat y))$
contains a non trivial segment $\{\hat y\}\times [0,a]$, $a>0$,
where $\hat y$ is a lift of $y$ in the central model.

Then for any $\theta>0$ and any neighborhood $U_\Lambda$ of $\Lambda$,
there is $\eta>0$ such that:
\begin{itemize}

\item considering the dynamics in a central model associated to the plaques $W^c_\eta$,
there is $x\in\Lambda$ having a chain-recurrent central segment $I$;

\item for any $z\in {\rm Int}(I)$ and for any neighborhood $V_z$ of $z$, there is a
periodic point $p\in V_z$ such that
\begin{itemize}

\item ${\rm Orb}(p)\subset U_\Lambda$.

\item $L^c(p)\in(-\theta,\theta)$.

\end{itemize}
\end{itemize}
\end{prop}

One can thus apply the Corollary~\ref{c.central-segment}. This shows that the second case of
Theorem~\ref{main2} holds.

\subsection{The thin trapped case}
We now consider the other case, i.e. the assumption of Proposition~\ref{p.unstable} does not hold.
We apply the following lemma in order to select such a $0$-CLE set. This will be proved in Section~\ref{s.trapped}.

\begin{lemma}\label{l.selection}
There exist some constants $C>0$, $\sigma\in (0,1)$, a sequence of $0$-CLE sets
$(K_n)$ and a sequence of points $(z_n)$ in $\Lambda$ such that:
\begin{itemize}
\item[(i)] for each $n$, the $\alpha$-limit set of $z_n$ is $K_n$,
\item[(ii)] the sequence $(z_n)$ converges toward a point of $X$,
\item[(iii)] for each $n$, the point $z_n$ is a $(C,\sigma,F)$-hyperbolic for $f^{-1}$.
\end{itemize}
\end{lemma}

The points of $X$ have a uniform stable manifold tangent to the bundle $E^{s}\oplus E^c$
and the points $y_n$ of the previous lemma have a uniform unstable manifold tangent to $F$.
As a consequence for $n$ large enough, the unstable manifold of $y_n$ meets the stable manifold of $X$.
So, the sets $K_n$ for $n$ large will be shown in Section~\ref{s.trapped}, to satisfy the following statement.

\begin{prop}\label{p.trapped}
Assume that for any $0$-CLE set $K_0$ and any point $y\in \Lambda$ such that
$\alpha(y)\subset K_0$ we have for any central model and any lift $\hat y$ of $y$,
$$\hat W^{ch-u}(\hat K_0\cup\orb^-(\hat y))\; \cap\; \left(\{\hat y\}\times [0,+\infty)\right)=\{(\hat y,0)\}.$$

Then, for any neighborhood $U_\Lambda$ of $\Lambda$,
there exists a $0$-CLE set $K\subset \Lambda$ and $\eta>0$ such that the plaques
$W^c_{\eta}(x)$ for $x\in K$ are contained in the chain-stable set $W^{ch-s}_{U_\Lambda}(\Lambda)$ of $\Lambda$ in $U_\Lambda$.
\end{prop}

Let $K$ be the set provided by the Proposition~\ref{p.trapped}.
If one considers a central model $(\hat K\times [0,+\infty),\hat f)$ associated to $(K,f)$,
the chain-unstable set of $\hat K\times \{0\}$ does not contain a
non-trivial interval $\{\hat y\}\times [0,a]$, $a>0$ by our assumptions.
The classification Theorem~\ref{cm}(1) implies that $(\hat K\times [0,+\infty),\hat f)$ is thin trapped. This proves that the central dynamics of $K$ is thin trapped.

The assumptions of Proposition~\ref{NS} are now satisfied by the set $K$: there exists a periodic orbit
contained in an arbitrarily small neighborhood of $K$ and both in the local chain-stable and chain-unstable sets
of $\Lambda$ inside a small neighborhood $U_\Lambda'$ of $\Lambda$. Since $K$ is a $0$-CLE set, the central exponents
of $p$ are contained in $(-\theta,\theta)$. By Corollary~\ref{c.local-class}, the local homoclinic class $H(p,U_0)$ contains $\Lambda$.
This shows that the second case of Theorem~\ref{main2} holds for $U_0$ and $\theta$.

\section{The chain-unstable case: proof of Proposition~\ref{p.unstable}}\label{s.connecting}
Let us fix $\theta>0$ and a small neighborhood $U_\Lambda$ of $\Lambda$.
By assumption, there exists a central model $(\hat f, \hat \Lambda)$, a point $\hat y\in \hat \Lambda$ and an invariant compact set $\hat K\subset \hat \Lambda$
such that:
\begin{itemize}
\item[--] $\hat K$ contains $\alpha(\hat y)$ and projects by $\pi\colon \hat \Lambda\times [0,+\infty)\to M$ as a $0$-CLE set $K\subset \Lambda$ of $f$.
\item[--] The chain-unstable set of $\hat y$ is non-trivial: $\hat W^{ch-u}(\hat K\cup \orb^-(\hat y))\setminus \{\hat y\}\neq
\emptyset$.
\end{itemize}

\paragraph{\bf Choice of $\eta$.} Let us choose $\eta>0$ small.
One can always reduce the central model to a neighborhood of $\hat \Lambda\times \{0\}$
such that it becomes a central model associated to the plaques $W^c_\eta$.
By a conjugacy of the form $(\hat x,t)\mapsto (\hat x,a.t)$,
one can ``normalize" the central model so that $\hat f$ is defined from
$\hat \Lambda \times [0,1]$ to $\hat \Lambda\times [0,+\infty)$.

Consider $\hat X$ the set of points of $\hat \Lambda$ whose projection by $\pi$ belongs to $X$.
There exists $\kappa$ such that for each $x\in X$ we have:
\begin{itemize}
\item[--] $f^k(W^c_\kappa(x))\subset W^c_\eta(f^k(x))$ for each $k\geq 0$,
\item[--] $W^c_\kappa(x)$ is contained in the stable set of $x$.
\end{itemize}
This implies the existence of $a>0$ such that for any $\hat x\in \hat X$ one has
\begin{itemize}
\item[--] $\hat f^k(\{\hat x\}\times [0,a])\subset \{\hat f^k(\hat x)\}\times [0,1]$ for each
$k\geq 0$,
\item[--] the length of $\hat f^k(\{\hat x\}\times [0,a])$ goes to $0$ as $k\to +\infty$.
\end{itemize}
Since the chain-unstable set of $\hat y$ is non-trivial,
one can reduce $a>0$ so that one has
$$\hat W^{ch-u}(\hat K\cup \orb^-(\hat y))\supset\{\hat y\}\times [0,a].$$

Choosing $\eta$ small, one can also require that the $\eta$-neighborhood of $\Lambda$ is contained in $U_\Lambda$
and for any $x\in \Lambda$ and  $z\in W^c_\eta(x)$ we have
$$\left| \log\|Df_{|E^c}(x)\| - \log\|Df_{|T_zW^c_\eta(x)}(z)\|\right|< \theta/4.$$
\bigskip

\paragraph{\bf The point $x$ and the segment $I$.} We apply the following lemma.

\begin{lemma}
There exists a sequence $(\hat x_n)$ in $\hat \Lambda$ such that for each $n$ we have:
\begin{itemize}
\item[--] $\alpha(\hat x_n)$ projects by $\pi$ on a $0$-CLE set,
\item[--] the set $\{\hat x_n\}\times [0,a]$ is contained in the chain-unstable set
of $\hat \Lambda\times \{0\}$,
\item[--] the set $\{\hat x_n\}\times [0,a]$ is contained in the $1/n$-chain-stable set of $\hat \Lambda\times \{0\}$.
\end{itemize}
In particular for each $k\geq 0$, we have
$$\hat f^{-k}(\{\hat x_n\}\times [0,1])\subset \{\hat f^{-k}(\hat x_n)\}\times [0,1].$$
\end{lemma}
\begin{proof}
Let us fix $n\geq 1$.
Let $\hat U_n$ be the $1/n$-neighborhood of $\hat X$.
For any $\ell\geq 1$, consider a $1/\ell$-pseudo-orbit $O_\ell=\{p_0,\dots,p_{n_\ell}\}$ from
$\hat y$ to a point $z_{\ell}$ of the closure $\hat U_n$ such that $O_\ell\setminus\{z_{\ell}\}$
does not intersect the interior of $\hat U_n$.
One can define continuous maps $s\mapsto t_k(s)$ for each $0\leq k\leq n_\ell$ such that
$t_0(s)=s$ and
for each $s>0$ small, the sequence $(p_k,t_k(s))$ is a $1/\ell$-pseudo-orbit between
$(p_0,s)$ and $(p_{n_\ell},t_{n_\ell}(s))$.

When $s$ increases some $t_k(s)$ becomes equal to $a$.
We denote $\hat x_{n,\ell}:=p_k$. By construction $\{\hat x_{n,\ell}\}\times [0,a]$
is in $1/\ell$-chain-unstable set of $\{\hat y\}\times [0,s]$, hence of $\hat \Lambda\times \{0\}$.
It is also in the $1/\ell$-chain-stable set of $\{z_{\ell}\}\times [0,a]$. Since $z_{\ell}$
is $1/n$-close to $\hat X$, the set $\{\hat x_{n,\ell}\}\times [0,a]$ is in the $1/n$-chain-stable
set of $\hat X\times [0,a]$, hence of $\hat \Lambda\times \{0\}$.

Let $\hat x_n$ be a limit of $\hat x_{n,\ell}$ when $\ell$ goes to $+\infty$.
Passing to the limit $\{\hat x_{n}\}\times [0,a]$ is in the chain-unstable set of
$\hat \Lambda\times \{0\}$ and in the $1/n$-chain-stable set of $\hat \Lambda\times \{0\}$.
By construction the set $\alpha(\hat x_n)$ is disjoint from the interior of
$\hat U_n$. Thus its projection $K$ by $\pi$ is a chain-transitive invariant compact set disjoint from $X$.
So it is a $0$-CLE set.
\end{proof}

We define the point $\hat x\in \hat \Lambda$ as a limit of the points $\hat x_n$.
We set $x=\pi(\hat x)$ and the segment $I$ is the projection by $\pi$ of
$\{\hat x\}\times [0,a]$. By construction, it is a chain-recurrent central segment of $x$
for the central model we considered.
\bigskip

\paragraph{\bf The point $z$.}
Let $z$ be an interior point of $I$.
Note that it is enough to prove the Proposition~\ref{p.unstable} for a dense subset of points $z\in I$.
Since $f$ is $C^1$-generic, it has at most countably many periodic points, hence one can assume that $z$ is
not periodic.

One can find for each $n$ a point $z_n\in W^c_\eta(x_n)$ such that:
\begin{itemize}
\item[--] the sequence $(z_n)$ converges to $z$,
\item[--] for each $k\geq 0$, the point $f^{-k}(z_n)$ belongs to $W^c_\eta(f^{-k}(x_n))$,
\item[--] $z_n$ belongs to the chain-unstable set of $\Lambda$ in the $\eta$-neighborhood of $\Lambda$.
\end{itemize}
The modulus of the central Lyapunov exponents of the ergodic measures supported in $\alpha(z_n)$ are smaller than $\theta/3$:
\begin{lemma}\label{l.exponent}
Let us consider $x\in \Lambda$ and $z\in M$ such that
for each $k\geq 0$, the iterate $f^{-k}(z)$ belongs to $W^c_\eta(f^{-k}(x))$ and
$\alpha(x)$ is a $0$-CLE set. Then, any ergodic measure supported on $\alpha(z)$ has a central Lyapunov exponents in
$(-\theta/3,\theta/3)$.
\end{lemma}
\begin{proof}
Consider any ergodic measure $\mu$ supported on $\alpha(z)$.
Since the bundle $E^c$ is one-dimensional, its central Lyapunov exponent is equal to
$$L^c(\mu)=\int \log(\|Df_{|E^c}\|)d\mu.$$
Hence there exists a segment of orbit
$(f^{-\ell}(z),\dots,f^{-k}(z))$, with $k$ large and $\ell>k$, such that
$$\frac 1 {\ell-k}\log \|Df_{|TW^c_\eta(f^{-\ell}(x))}^{\ell-k}(f^{-\ell}(z))\|$$ is close to $L^c(\mu)$.
By our choice of $\eta$, this last quantity is $\theta/4$-close to
$$\frac 1 {\ell-k}\log \|Df_{|E^c}^{\ell-k}(f^{-\ell}(x))\|,$$
which is arbitrarily close to zero if $k$ and $\ell-k$ are large since $\alpha(x)$ is a $0$-CLE set.
\end{proof}
\bigskip

\paragraph{\bf Conclusion of the proof of the Proposition~\ref{p.unstable}.}
Let us consider a neighborhood $V_z$ of $z$.
Note that $U_\Lambda$ and $V_z$ can be chosen in the countable basis $\cB$ introduced at Section~\ref{ss.generic}.
For any $C^1$-neighborhood $\cU$ of $f$, the following lemma will provide a diffeomorphism
$g\in \cU$ having a hyperbolic periodic orbit $O\subset U_\Lambda$ which meets $V_z$ and whose
$(\dim(E^s)+1)^\text{th}$ Lyapunov exponent belongs to $(-\theta,\theta)$.
Since $f$ belongs to the dense $G_\delta$ set $\cG$, then the same property holds for $f$,
ending the proof of the Proposition~\ref{p.unstable}. \qed

\begin{lemma}\label{l.perturbation}
For any $C^1$-neighborhood $\cU$ of $f$, there exists $g\in \cU$ having a periodic orbit
$O\subset U_\Lambda$ which meets $V_z$ and whose
$(\dim(E^s)+1)^\text{th}$ Lyapunov exponent belongs to $(-\theta,\theta)$.
\end{lemma}

\paragraph{\bf The perturbation argument.} We give now the proof of the Lemma~\ref{l.perturbation}.
We can assume that there is $r_0>0$ such that $\alpha(z_n)\cap B(z,r_0)=\emptyset$ for any $n$
and that in particular $z_n\notin \alpha(z_n)$ for $n$ large.
Since otherwise, the fact $f$ is $C^1$ generic implies that $\alpha(z_n)$ can be accumulated by periodic orbits
for the Haudorff topology.
Since the central Lyapunov exponent of any ergodic measure supported on $\alpha(z_n)$ belongs to $(-\theta/3,\theta/3)$, we know that the central Lyapunov exponent of these periodic orbits belongs to $(-\theta,\theta)$. Thus, we know that the proposition is satisfied in this case.
\medskip

\paragraph{\it The neighborhood $\cU$.}
One can always reduce the neighborhood $\cU$ and find $N_0\geq 1$ so that
for any diffeomorphism $g\in \cU$, for any point $x$ whose $g$-orbit is contained in $U_\Lambda$
and any point $z\in \Lambda$, such that $g^k(x)$ and $f^k(z)$ are close for each $|k|\leq N_0$, then
the spaces $E^c_x$ for $g$ and $E^c_z$ for $f$ are close.
Consequently, $\|Dg_{|E^c}(x)\|$ and $\|Df_{|E^c}(z)\|$ are $\theta/3$-close.

For the neighborhood $\mathcal U$, we know there is $L\in\NN$ associated to $\mathcal U$ by Lemma~\ref{l.connecting}.

\medskip

\paragraph{\it The perturbation domain at $z$.}
By Lemma \ref{l.connecting}, one can associate two neighborhoods $B_z\subset \hat B_z$ at $z$, small enough so that:
\begin{itemize}
\item[--] $\hat B_z\subset V_z$;
\item[--] $U_{L,z}:=\bigcup_{i=0}^{L-1}f^i(\hat B_z)$ is disjoint from $\alpha(z_n)$ for any $n\in\NN$
(using that $\alpha(z_n)\cap B(z,r_0)=\emptyset$);
\item[--] $U_{L,z}\subset U_\Lambda$.
\end{itemize}
\medskip

\paragraph{\it The set $A$ and the connecting orbit between $A$ and $z$.}
Choose $n_0\in\NN$ such that $z_{n_0}\in B_z$. We denote $A:=\alpha(z_{n_0})$.
Choose a small neighborhood $U_A$ of $A$ such that:
\begin{itemize}
\item[--] $\overline{U_A}\subset U_\Lambda$.

\item[--] $U_A\cap U_{L,z}=\emptyset$.

\item[--] There are $N_1\geq 1$ and $\delta>0$ such that for
any diffeomorphism $g\in \cU$ whose $C^0$-distance to $f$ on $U_A$ is smaller than $\delta$ and
for any orbit segment $\{x,g(x),\cdots,g^{N_1}(x)\}\subset U_A$ for $g$
of a point $x$ whose orbit is contained in $U_\Lambda$, then
$$-\frac 2 3 \theta<\frac{1}{N_1}\sum_{i=0}^{N_1-1}\log\|Dg|_{E^c(g^i(x))}\|<\frac 2 3 \theta.$$
\end{itemize}
The last item can be justified as follows. If $U_A$ is small and $g$ is $C^0$-close to $f$ on $U_A$, then any orbit segment
$\{x,g(x),\cdots,g^{N_1}(x)\}$ for $g$ is arbitrarily close to an orbit segment
$\{z,f(z),\cdots,f^{N_1}(z)\}$ for $f$, hence by our choice of $\cU$ we have
$$\frac{1}{N_1}\sum_{i=0}^{N_1-1}\log\|Dg|_{E^c(g^i(x))}\|\; - \; \frac{1}{N_1}\sum_{i=0}^{N_1-1}\log\|Df|_{E^c(f^i(z))}\|\in (-\theta/3,\theta/3).$$
If $N_1$ has been chosen large enough, the second sum belongs to $(-\theta/3,\theta/3)$ by the Lemma~\ref{l.exponent} above.
\medskip

\paragraph{\it The points $a$ and $b$.}
Since $z$ and $\Lambda$ are in a same chain-transitive set contained in $U_\Lambda$ and since $A=\alpha(z_{n_0})$ is in the local chain-unstable set of $\Lambda$ in the $\eta$-neighborhood of $\Lambda$, together with the fact that $f$ is $C^1$ generic, by Theorem \ref{t.weaktransitive}, there is a sequence of orbit segments $$\{y_{n},f(y_n),\cdots,f^{k_n}(y_n)\}_{n\in\NN}\subset U_\Lambda$$
such that
\begin{itemize}

\item[--] $\lim_{n\to\infty}y_n=z$,

\item[--] $\lim_{n\to\infty}f^{k_n}(y_n)=a\in A$.

\end{itemize}

Let us consider a smaller neighborhood $U'_A$ of $A$ whose closure is contained in the interior of $U_A$.
For each $n$, define $l_n\in[0,k_n]\cap\NN$ such that
\begin{itemize}

\item[--] $f^{l_n-1}(y_n)\notin U'_A$,

\item[--] $f^m(y_n)\in U'_A$ for any $l_n\le m\le k_n$.

\end{itemize}
By taking a subsequence if necessary, one can assume that $\lim_{n\to\infty}f^{l_n}(y_n)=b\in\overline{U'_A}$.
Moreover, the forward iterations of $b$ are contained in $U_A$.
\medskip

\paragraph{\it The perturbation domain at $b$ and the connecting orbit between $z$ and $b$.}
By Lemma \ref{l.connecting}, there are two neighborhoods $B_b\subset \hat B_b$ at $b$ small enough so that:
\begin{itemize}
\item[--] the connected components of $U_{L,b}$ are smaller than $\delta$;
\item[--] $U_{L,b}$ is disjoint from $A$ and from the backward orbit of $z_{n_0}$;

\item[--] $U_{L,b}\subset U_A$.

\end{itemize}

Choose $n_1$ such that $y_{n_1}\in B_z$ and $f^{l_{n_1}}(y_{n_1})\in B_b$.
We also choose $m_1\geq 1$ so that the backward orbit of $f^{-m_1}(z_{n_0})$ is contained in $U_A$.
Then we define $\tau=\max(m_1,l_{n_1})$.
We introduce an integer $T$ much larger than $\tau$.
\medskip

\paragraph{\it The perturbation domain at $a$ and the connecting orbit between $b$ and $a$.}
By Lemma \ref{l.connecting}, there are two neighborhoods $B_a\subset \hat B_a$ at $a$ small enough so that:
\begin{itemize}

\item[--] The connected components of $U_{L,a}$ are smaller than $\delta$;

\item[--] $U_{L,a}\subset U_A$;

\item[--] $U_{L,a}$ is disjoint from $U_{L,b}$;

\item[--] the points $f^{-m_1-k}(z_{n_0})$ for $0\leq k\leq T$ are not in $\hat B_a$.

\end{itemize}

Choose $n_2$ such that $f^{l_{n_2}}(y_{n_2})\in B_b$ and $f^{k_{n_2}}(y_{n_2})\in B_a$.
Note that the orbit segment $\{f^{l_{n_2}}(y_{n_2}), f^{k_{n_2}}(y_{n_2})\}$ is contained in $U_A$.
\medskip

We now realize successively three perturbations of $f$ in $\cU$ supported in the disjoint domains $U_{L,z},U_{L,b},U_{L,a}$ given by the Lemma~\ref{l.connecting}.
The composition of the three perturbations provides a diffeomorphism $g$ which also belongs to $\cU$
(see~\cite[remark 4.3]{BC}).
\medskip

\paragraph{\it The perturbation at $z$.}
We apply the Lemma \ref{l.connecting} to the perturbation domain $B_z,\hat B_z$
and the points $f^{-m_1}(z_{n_0})$ and $f^{l_{n_1}}(y_{n_1})$. We obtain a diffeomorphism $f_1\in \cU$
and a point $p\in \hat B_z\subset V_z$ whose orbit
contains $f^{-m_1}(z_{n_0}), f^{l_{n_1}}(y_{n_1})$ and in particular satisfies:

\begin{itemize}

\item[--] there exists a forward iterate $f_1^k(p)$ in $B_b$ for some $k\leq 2\tau$
and any point $f_1^i(p)$ with $0\leq i \leq k$ belongs to $U_\Lambda$,
\item[--] the backward orbit of $p$ for $f_1$ is contained in $U_\Lambda$, is disjoint from $U_{L,b}$ and its $\alpha$-limit is $A$,
\item[--] the backward orbit of $f_1^{-2\tau}(p)$ is contained in $U_A$,
\item[--] the backward iterates $f_1^{-k}(p)$ for any $0\leq k\leq T$ are not in $\hat B_a$,
\item[--] $f_1$ and $f$ coincide on $U_A$.
\end{itemize}
\medskip

\paragraph{\it The perturbation at $b$.}
We apply the Lemma \ref{l.connecting} to the perturbation domain $B_b,\hat B_b$
and the points $p$ and $f^{k_{n_2}}(y_{n_2})$ for the map $f_1$ which coincides with $f$
on $U_{L,b}$. We obtain a diffeomorphism $f_2\in \cU$ such that:

\begin{itemize}

\item[--] there exists a forward iterate $f_2^m(p)$ in $B_a$ for some $m\geq 0$,
the iterates $f_2^k(p)$ for $0\leq k \leq m$ are contained in $U_\Lambda$ and the iterates $f_2^k(p)$ for $2\tau \leq k \leq m$ are contained in $U_A$;
\item[--] the backward orbit of $f_2^{-2\tau}(p)$ is contained in $U_A$ and its $\alpha$-limit is $A$;
\item[--] the backward iterates $f_2^{-k}(p)$ for any $0\leq k\leq T$ are contained in $U_\Lambda\setminus \hat B_a$,
\item[--] $f_2$ and $f$ coincide on $U_{L,a}$.
\end{itemize}
\medskip

\paragraph{\it The perturbation at $a$.}
We apply the Lemma \ref{l.connecting} to the perturbation domain $B_a,\hat B_a$
and to the forward and backward orbit of $p$ for the map $f_2$ which coincides with $f$
on $U_{L,a}$. We obtain a diffeomorphism $f_3$ such that:

\begin{itemize}

\item[--] $p$ is periodic and its orbit is contained in $U_\Lambda$,
\item[--] the period is larger than $T$,
\item[--] all the iterates $f^k(p)$ different from
$\{f^{-4\tau}(p),\dots,p,\dots f^{4\tau}(p)\}$ are in $U_A$.
\end{itemize}
\medskip

\paragraph{\it End of the proof of the lemma.}
By an arbitrarily small perturbation which preserves the orbit of $p$ by $f_3$,
one gets a diffeomorphism $g\in \cU$ such that the orbit of $p$
is hyperbolic. By construction it is contained in $U_\Lambda$ and meets $V_z$.

Note that $g$ is $\delta$-close to $f$ for the $C^0$-distance on $U_A$.
If $P$ is the period of $p$,
one deduces that
$$-\frac 2 3 \theta<\frac{1}{P-8\tau}\sum_{i=4\tau}^{P-4\tau-1}\log\|Dg|_{E^c(g^i(x))}\|<\frac 2 3 \theta.$$
Since $P>T$ is much larger than $\tau$, one deduces that
$$\frac{1}{P}\sum_{i=0}^{P-1}\log\|Dg|_{E^c(g^i(x))}\|\in (-\theta,\theta).$$
Hence the central Lyapunov exponent of $p$ has a modulus smaller than $\theta$ and the Lemma~\ref{l.perturbation} is now proved. \qed

\section{The thin trapped case: proof of Lemma~\ref{l.selection} and Proposition~\ref{p.trapped}}\label{s.trapped}
We first prove the Lemma~\ref{l.selection} which allows to select the set $K$.

\begin{proof}[\bf Proof of Lemma~\ref{l.selection}]
We first build two sequences $(y_n)$ and $(K_n)$ satisfying the properties (i) and (ii) of the lemma.
For that, we will divide the proof into two cases: the non-isolated case and
the isolated case.

\begin{itemize}

\item Non-isolated case: there is a sequence of 0-CLE sets $\{K_n\}$
such that there is $y_n\in K_n$ for each $n$,
$y=\lim_{n\to\infty}y_n$ exists and $y\in X$.

\item Isolated case: there is a neighborhood $U$ of $X$ such that
for any 0-CLE set $K$, one has $K\cap U=\emptyset$.
\end{itemize}

In the non-isolated case, (i) and (ii) are satisfied by $(K_n)$, $(y_n)$ and $y$.

Now we discuss the isolated case. For any neighborhood $U$ of $X$,
let $K_{\Lambda\setminus U}=\cap_{n\in\ZZ}f^n(\Lambda\setminus U)$ be the maximal invariant
set in $\Lambda\setminus U$.
By Lemma~\ref{l.0CLE}, the central Lyapunov exponent of any ergodic measure supported on $K_{\Lambda\setminus U}$,
is 0. Take a sequence of neighborhoods
${U}_n$ of $X$ such that $\cap_{n\in\NN}\overline{U_n}=X$ and fix
$n$. For any $m\in\NN$, since $\Lambda$ is chain-transitive and
$\Lambda$ contains a 0-CLE set $K$, there is a $1/m$-pseudo orbit
from $K$ to $X$. As a consequence, there is a $1/m$-pseudo orbit
$\{z_a\}_{a=0}^{b(m)}$ such that
\begin{itemize}
\item[--] $z_{b(m)}\in U_n$ and $z_a\notin U_n$ for any $0\le a\le b(m)-1$.

\item[--] $z_0$ is $1/m$-close to $K$. Thus, $b(m)\to\infty$ as
$m\to\infty$.

\end{itemize}

By taking a subsequence if necessary, one can assume that
$y_n=\lim_{m\to\infty}z_{b(m)}$ exists and moreover,
\begin{itemize}

\item[--] $y_n\in \overline{U_n}$.

\item[--] ${\rm Orb}^-(y_n)\subset \Lambda\setminus U_n$. Hence, $\alpha(y_n)\subset \Lambda\setminus
U_n$. By the isolated assumption, $\alpha(y_n)\subset
K_{\Lambda\setminus U_n}$ is a 0-CLE set.

\end{itemize}
By taking a subsequence if necessary, one can assume that
$\lim_{n\to\infty}y_n$ exists and belongs to $X$.
This gives (i) and (ii) in the isolated case.
\medskip

Now, for both cases (the non-isolated case and the isolated case),
one gets: a sequence of $0$-CLE sets $(K_n)$ and a sequence of
points $(y_n)$ in $\Lambda$ such that:
\begin{itemize}
\item[--] for each $n$, the $\alpha$-limit set of $y_n$ is $K_n$,
\item[--] the sequence $(y_n)$ converges toward a point of $X$.
\end{itemize}
In particular we have
$$\limsup_{k\to\infty}\sum_{i=0}^{k-1}\frac{1}{k}\log\|Df^{-1}|_{E^c(f^{-i}(y_n))}\|<\log(\lambda ^{-1/2}).$$

We then explain how to get the points $(z_n)$ from the sequence $(y_n)$ in order to
satisfy (iii). We choose $\sigma\in (\lambda^{1/2},1)$, we set $\rho=\lambda^{1/2}$ and apply Lemma~\ref{l.hyperbolic-time}.
In the first case of the Lemma~\ref{l.hyperbolic-time}, the conclusion of Lemma~\ref{l.selection} holds directly
with the points $z_n=y_n$.

In the second case of the Lemma~\ref{l.hyperbolic-time}, the point $z_n$ are $(1,\sigma,F)$-hyperbolic and
(taking a subsequence) converge toward a point $z$ that is $(1,\rho,E^c)$-hyperbolic. Since
$\rho=\lambda^{1/2}<e^{-\epsilon}$, the point $z$ belongs to $X$ as required. Thus $\{z_n\}$ satisfies all the
requirements of the Lemma~\ref{l.selection}.
\end{proof}
\bigskip

Now we will discuss the thin trapped case.

\begin{proof}[\bf Proof of Proposition \ref{p.trapped}]
The Lemma~\ref{l.selection} gives us some constants $C>0$ and $\sigma\in (0,1)$.
We choose any $\rho\in (\sigma,1)$ and a constant $r>0$ given by Lemma~\ref{l.manifold2}
adapted to the constants of hyperbolicity $C,\rho$.

We choose $\chi>0$ and work with central plaques $W^c_\chi$.
We require several conditions on $\chi$:
\begin{itemize}
\item[--] $\chi<r/2$,
\item[--] for any $x\in\Lambda$ the plaque $W^c_\chi(x)$ is contained in $U_\Lambda$,
\item[--] for any $(C,\sigma,F)$-hyperbolic point $y\in \Lambda$ for $f^{-1}$
then any point $y'$ such that $f^{-n}(y')$ belongs to $W^c_\chi(f^{-n}(y))$ for each $n\geq 0$
is $(C,\rho,F)$-hyperbolic for $f^{-1}$.
\end{itemize}

By Lemma \ref{l.selection}, one can choose a 0-CLE set $K$
and $y\in \Lambda$ such that
\begin{itemize}

\item[--] $\alpha(y)=K$.

\item[--] There is $x\in X$ such that $d(x,y)<r/2$.

\item[--] $y$ is $(C,\sigma,F)$-hyperbolic.

\end{itemize}

Let us consider any central model $(\hat f,\hat \La\times [0,+\infty))$ of $\La$ associated to the plaques $W^c_\chi$
and let $\hat y$ be any lift of $y$ in the central model.
Then $\hat K=\alpha(\hat y)$ is a lift of $K$.

\begin{claim}

There is $\delta_{\hat K}>0$ such that $\pi(\hat K\times [0,\delta_{\hat K}])$ is contained
in the chain-stable set $W^{ch-s}_{U_\Lambda}(\Lambda)$ of $\Lambda$ in $U_\Lambda$.

\end{claim}

\begin{proof}[Proof of the claim]
One can assume that the map $\hat f$ is defined from
$\hat \La\times [0,1]$ to $\hat \La\times [0,+\infty)$.
There are two possibilities:
\begin{enumerate}

\item There is a sequence $\{\delta_n\}$ and a sequence $\{k_n\}\subset \NN$ such that
\begin{itemize}
\item[--] $\delta_n\to 0$ as $n\to\infty$,

\item[--] $f^{-k}(\{\hat y\}\times [0,\delta_n])\subset \{\hat f^{-k}(\hat y)\}\times [0,1]$ for any $0\le k\le k_n$,
\item[--] $\hat f^{-k_n}(\{\hat y\}\times [0,\delta_n])=\{\hat f^{-k_n}(\hat y)\}\times [0,1]$.

\end{itemize}

\item There is $\delta_{\hat y}>0$ such that $\hat f^{-n}(\{\hat y\}\times [0,\delta_y])\subset
\{\hat f ^{-n}(\hat y)\}\times [0,1]$ for any $n\in\NN$.

\end{enumerate}
In case (1), the chain-stable set of $\hat K\times \{0\}$ contains
an interval $\{\hat z\}\times [0,1]$.
By our assumptions, the chain-unstable set does not contain a non-trivial interval
$\{\hat y'\}\times [0,1]$, hence the central dynamics of $\hat K\times \{0\}$
is thin trapped by the theorem~\ref{cm}(1).
Arguing as in the begining of the proof of Proposition~\ref{NS},
one deduces that there exists $a>0$ such that
$\hat K\times [0,a]$ is in the chain-stable set of $\hat K\times \{0\}$ in the central model.
Since the plaques $W^{c}_\chi$ are contained in $U_\Lambda$, the projection
$\pi(\hat K\times [0,1])$ is contained in $W^{ch-s}_{U_\Lambda}(\La)$.

In case (2), the length of $\hat f^{-n}(\{\hat y\}\times [0,\delta_y])$
is bounded away from zero, since otherwise $\{\hat y\}\times [0,\delta_y]$ would be contained in the
chain-unstable set of $\hat y$ in the central model
and this would contradict our assumption that $\hat W^{ch-u}(\hat y)$ is reduced to $\{\hat y\}$.
We let $\delta_{\hat K}$ be a lower bound for the length of
$\hat f^{-n}(\{\hat y\}\times [0,\delta_y])$.

Let us consider $(\hat x,a)\in \hat K\times [0,\delta_{\hat K}]$ and fix $\alpha>0$.
By construction, there exists $n\geq 0$ arbitrarily large and $t\in [0,\delta_y]$ such that
$y':=\pi(\hat f^{-n}(\hat y,t)))$ is $\alpha$-close to $\pi(\hat x,a)$.
Since $\hat f^{-n}(\hat y,t)$ are defined for any $n\geq 0$, the backward iterates
of $y'$ belong to the plaques $f^{-n}(W^{c}_\chi(y))$.
In particular $y'$ belongs to $\Lambda^-$ and by our choice of $\chi$ and $y$ we have
$d(y',x)<r$. We can thus apply the Lemmas~\ref{l.manifold1} and~\ref{l.manifold2}:
there exists a point $z$ whose orbit is contained in $U_\Lambda$ such that
$d(f^k(z), \Lambda)$ goes to zero as $k$ goes to $+\infty$ and such that
$d(f^{-n}(z), f^{-n}(y'))$ is exponentially small in $n$. If $n$ has been chosen large enough,
$f^{-n}(z)$ is at distance smaller than $2\alpha$ from $\pi(\hat x,a)$.
Since $\alpha$ is arbitrarily small, the point $\pi(\hat x,a)$ belongs to
$W^{ch-s}_{U_\Lambda}(\Lambda)$.
\end{proof}

By changing the lift of $\hat y$ (in the case the central dynamics of $\Lambda$ is non-orientable)
or by changing the central model, we get another lift $\hat K$, another constant $\delta'_{\hat K'}$
such that the union of the projections $\pi(\hat K\times [0,\delta_{\hat K}])\cup
\pi(\hat K'\times [0,\delta'_{\hat K'}])$ cover the plaques $W^c_\eta(x)$ with $x\in K$
for some $\eta>0$ small.
\end{proof}

\section{Consequences}\label{s.consequences}
In this section we prove the corollaries stated in the introduction.

\subsection{Central bundles of homoclinic classes. Proof of Corollary~\ref{c.ergodic}}
\label{ss.central}
Let us consider a homoclinic class $H(p)$ for a $C^1$-generic diffeomorphism $f\in \diff^1(M)\setminus \HT$, and
a one-dimensional central bundle $E^c$. Note that if $H(p)$ contains periodic orbits such that $E^c$ is stable
and others such that $E^c$ is unstable, then, the class contains a chain-transitive central segment associated
to $E^c$. Otherwise, the next result shows that the central dynamics along the bundle $E^c$ is thin trapped for
$f$ or $f^{-1}$ (implying that $E^c$ has type (H)-attracting or (H)-repelling according to the classification of
\cite[Section 2.2]{Cro08} and that it is ``chain-hyperbolic", according to the definition introduced in
\cite[Definition 7]{CP}).

\begin{prop}\label{p.TT}
Let $f$ be a diffeomorphism in a dense $G_\delta$ subset of $\diff^1(M)$
and let $H(p)$ be a homoclinic class endowed with a dominated splitting
$T_{H(p)}M=E\oplus E^c\oplus F$ such that:
\begin{itemize}
\item[--] the minimal index of $H(p)$ equals $\operatorname{dim}(E\oplus E^c)$,
\item[--] there exists periodic orbits related with $p$ whose Lyapunov exponent along $E^c$ is arbitrarily close to $0$.
\end{itemize}
Then, the central dynamics along $E^c$ is thin-trapped.
In particular $E$ is uniformly contracted.
\end{prop}
\begin{proof}
Note that since $H(p)$ contains periodic orbits such that $E^c$ is stable, there is no central model associated
to $E^c$ such that the dynamics is trapped for $\hat f^{-1}$. According to Theorem~\ref{cm}(2) we have to rule out
the existence of a chain-recurrent central segment $I$ for $H(p)$. Since $I\subset H(p)$, for any $z$ in the
interior of $I$ there exists a sequence of periodic points $p_n$ converging to $z$ and whose central Lyapunov
exponent is close to zero. Since $f$ is $C^1$-generic, by Lemma~\ref{l.index} one can choose the points $p_n$ with index equal to
$\operatorname{dim}(E)$. By Corollary~\ref{c.central-segment} the points $p_n$ belong to $H(p)$ and we obtain a
contradiction.
\end{proof}

We prove the result about the ergodic closing lemma.
\begin{proof}[Proof of Corollary~\ref{c.ergodic}]
By Theorem~\ref{main}, there exists a dominated splitting
$T_{H(p)}M=E^s\oplus E^c\oplus F$ with $\dim(E^c)=1$
and $\dim(E^s\oplus E^c)=i$.
Moreover, there exists periodic orbits in $H(p)$ whose Lyapunov exponent along $E^c$
is close to zero.
By the previous lemma the central dynamics along $ E^c$ is thin trapped.

By Ma\~n\'e's ergodic closing lemma~\cite{mane-ergodic,abc-mesure},
there exists a sequence of periodic orbits $\bar O_n$ whose associated measures
converge towards $\mu$. Note that they are contained in an arbitrarily small neighborhood of $H(p)$.
Consider a plaque family $\D^{cs}$ tangent to $E^{s}\oplus E^c$ over the maximal invariant set
in a small neighborhood of $H(p)$, which is trapped by $f$ and whose plaques are small.
For some $\sigma\in (0,1)$, there exists a $(1,\sigma,F)$-hyperbolic point $x$
in the support of $\mu$. Its unstable manifold meets the plaque $\D^{cs}_{\bar q_n}$
of some point $\bar q_n\in \bar O_n$ for large $n$.
Since the plaques are trapped, it meets the stable manifold of some hyperbolic periodic point $q_n\in \D^{cs}_{\bar q_n}$.
This implies that $q_n$ belongs to the chain-unstable set of $H(p)$.
Considering smaller plaque families,
this construction shows that there exists a sequence of hyperbolic periodic orbits $O_n$
whose associated measures converge to $\mu$ and that are contained in the chain-unstable set of $H(p)$.

On the other hand, since the central exponent of $O_n$ is close to zero,
some point $q_n'\in O_n$ has a uniform unstable manifold tangent to $F$.
One can assume that $(q'_n)$ converges toward some point $x\in H(p)$.
Since the central dynamics along $E^c$ is thin trapped, the Lemma~\ref{basic}.
implies that the orbits $O_n$ also belong to the chain-stable set of $H(p)$.
As a consequence, $H(p)$ contains periodic orbits whose measure converge towards $\mu$.
\end{proof}

\subsection{More arguments about aperiodic classes}\label{ss.aperiodic}

We provide now an alternative argument for the item (1) of Theorem~\ref{main} about aperiodic classes, which is
shorter since it does not use Theorem~\ref{main2} (nor Lemma~\ref{l.class-domina}) and the Sections~\ref{s.main2}
to~\ref{s.trapped}.

\begin{prop}\label{p.aperiodic}
For any diffeomorphism $f$ in a dense $G_\delta$ subset of $\diff^1(M)\setminus \HT$, any aperiodic class $\C$
is partially hyperbolic with a one-dimensional central bundle.
\end{prop}
\begin{proof}
By Theorem~\ref{t.hausdorff}, the class $\C$ is limit of a sequence of hyperbolic periodic orbit
$O_n$ for the Hausdorff distance.
\medskip

Let us assume by contradiction that these periodic orbits do not have a weak Lyapunov exponent.
By Theorem~\ref{wen}, this induces a dominated splitting $T_\C M=E\oplus F$
where $\operatorname{dim}(E)$ equals the index of the periodic orbits
and there exist $C>0$, $\sigma\in (0,1)$ and $N\geq 1$ such that inequalities~\eqref{e.wen} in Theorem~\ref{wen} hold
for any point of the orbits $O_n$. We will set $\lambda=1/2$
and we can assume that $\sigma^2>\lambda=1/2$.

Passing to the limit, this implies that $\C$ has a $(C,\sigma,E)$-hyperbolic point $y$ for $f^N$. Let us
consider a sequence of periodic points $y_k$ in the union of the orbits $O_n$ which converges toward $y$. By
Lemma~\ref{l.hyperbolic-time}, one can replace the points $y_k$ and the point $y$ so that the first are
$(C,\sigma,F)$-hyperbolic for $f^{-N}$ and the second is $(C,\sigma, E)$-hyperbolic for $f^N$.
Lemma~\ref{l.manifold-intersect} implies that for $n$ large enough, the periodic orbits $O_n$ belong to the
chain-stable set of $\C$.

The symmetric argument for $f^{-1}$ shows that for $n$ large enough, the periodic orbits $O_n$ belong to the
chain-unstable set of $\C$. This proves that the $O_n$ are contained in $\C$, contradicting the assumption that
$\C$ is aperiodic.
\medskip

We have shown that the periodic orbits $O_n$ have a weak Lyapunov exponent, hence induce by Theorem~\ref{wen} a
splitting $E\oplus E^c\oplus F$ on $\C$, with $\operatorname{dim}(E^c)=1$. If the conclusion of the
Proposition~\ref{p.aperiodic} does not hold, either $F$ is not uniformly expanded or $E$ is not uniformly
contracted. One can assume for instance that $F$ is not uniform. With the domination, this implies that $\C$
contains a $E^c$-hyperbolic point for $f$. In particular, for any central model, the dynamics can not be thin
trapped for $f^{-1}$. By Corollary~\ref{c.central-segment}, the set $\C$ has no chain-recurrent central segment.
Thus by Theorem~\ref{cm}(2), the central dynamics of $\C$ is thin trapped for $f$. With the domination, we
conclude that $E$ is uniformly contracted.
The existence of a $E^c$-hyperbolic point for $f$ in $\C$
also implies that there exists a central plaque of $\C$ contained in the chain-stable set of $\C$.

Passing to the limit with the periodic orbits $O_n$, we also deduce that there exists a
$(C,\sigma,F)$-hyperbolic point for $f^{-N}$ in $\C$. Thus the result stated in Remark~\ref{r.related} applies
concluding that $\C$ contains a periodic point, which is a contradiction. The proof is now complete.
\end{proof}

\subsection{Weak periodic points: proof of Corollary~\ref{c.noweak}}

By Theorem~\ref{main}, for any $C^1$ generic $f$ which is away from ones exhibit a homoclinic tangency, every homoclinic class $H(p)$ admits a
dominated splitting $E\oplus F$ such that $\operatorname{dim}(E)$ equals the index of $p$.
If $H(p)$ is not hyperbolic, either $E$ is not uniformly contracted or
$F$ is not uniformly expanded.
We will assume that we are in the first case.
Theorem~\ref{main} gives a dominated splitting $E=E'\oplus E^c$
with $\dim(E^c)=1$.

If there exist periodic points in $H(p)$ of index $\operatorname{dim}(E')$,
then by Theorem~\ref{indices} there exists periodic orbits in $H(p)$ with the same index as $p$
whose Lyapunov exponent along $E^c$ is arbitrarily weak.
By Corollary~\ref{c.class}, these periodic points are homoclinically related to $p$
as required.

Otherwise the index of $p$ coincides with the minimal index of the class.
By Theorem~\ref{main}, there exists periodic orbits in $H(p)$ with the same index as $p$
whose Lyapunov exponent along $E^c$ is arbitrarily weak.
We conclude in the same way.

\subsection{About Palis conjecture: proof of Corollary~\ref{c.cycle}}
For a $C^1$ generic $f$ which is away from ones exhibit a homoclinic tangency or a heterodimensional cycle, Theorem~\ref{main} applies and
any homoclinic class $H(p)$ has a partially hyperbolic splitting
$T_{H(p)}M=E^s\oplus E^c_1\oplus\dots\oplus E^c_k\oplus E^u$.
If $k\leq 2$ we are done, otherwise the class contains periodic points of different indices
and Theorem~\ref{t.diffind} gives a contradiction.

\subsection{Lyapunov stable classes: proof of Corollary~\ref{c.lyap-stable}}
Consider a $C^1$-generic diffeomorphism $f$
in $\diff^1(M)\setminus \HT$ and a Lyapunov stable chain-recurrence class
$\C$.

\begin{claim}[\cite{yang-lyapunov}] The class $\C$ is a homoclinic class.
\end{claim}
Indeed if it is an aperiodic class, by Theorem~\ref{main} it has
a partially hyperbolic decomposition $T_\C M=E^s\oplus E^c\oplus E^u$
and this contradicts Theorem 5 in~\cite{bonatti-gan-wen}.
\medskip

\begin{claim}[\cite{yang-lyapunov}] If $T_{\C}M=E^s\oplus E^c_1\oplus \dots
\oplus E^c_k\oplus E^u$ is the partialy hyperbolic structure on $\C$,
then the class $\C$ contains periodic points of stable index
$\dim(E^s)+k$. In particular if it is not a sink, the bundle $E^u$ is non-degenerated.
\end{claim}
\begin{proof}
The proof is by contradiction.
Since it is similar to~\cite[Corollary 2.3]{CP}, we only give the idea.
Since $f$ is $C^1$-generic and $\C$ contains periodic points of index
$\dim(E^s)+k-1$ with Lyapunov exponent along $E^c_k$ arbitrarily close to zero.
Thus by Lemma~\ref{l.index}, there exist
periodic points $q$ of index $\dim(E^s)+k$
arbitrarily close to $\C$.
By Proposition~\ref{p.TT} the central dynamics along $E^c_k$ is thin trapped by $f^{-1}$.
One deduces that there exists a (small) plaque family
$\D^{cu}$ tangent to $E^{cu}=E^c_k\oplus E^u$ that is trapped by $f^{-1}$.
One also fix a (small) plaque family $\D^{cs}$ tangent to
$E^{cs}=E^s\oplus E^c_1\oplus \dots\oplus E^c_k$.
All the Lyapunov exponents of $q$ along $E^{cs}$ are uniformly bounded away from zero.
One deduces from Lemmas~\ref{l.infinite-pliss} and~\ref{l.manifold1} that up to replace $q$ by one of its iterates,
the stable manifold of $q$ contains $\D^{cs}_q$.
Similarly, $\C$ contains a dense subset of periodic points $x$ whose Lyapunov exponents
along $E^{cu}$ are uniformly bounded away from zero, implying that $\D^{cu}_{f^k(x)}$ is contained in $W^u(f^k(x))$ for some iterate $f^k(x)$ of $x$. The trapping property
on $\D^{cu}$ implies that $\D^{cu}_{x}$ is also contained in $W^u(x)$.
Since $q$ is close to $\C$, there exists such a point $x\in \C$
such that $\D^{cs}_q$ intersects $\D^{cu}_x$, hence $q$ belongs to
the closure of the unstable set of $x$.
Using that $\C$ is Lyapunov stable one deduces that $\C$ contains the periodic point
$q$. This contradicts the fact that $\C$ does not contain
any periodic point of index $\dim(E^s)+k$.
\end{proof}
\medskip

We now prove the two corollary.
\smallskip
\begin{proof}[Proof of Corollary~\ref{c.lyap-stable}]
We will assume that a Lyapunov stable class $H(p)$ has no robust heterodimensional cycle.
Arguing as in the previous section and using the previous claim, $H(p)$ has
a dominated splitting $T_{H(p)}M=E^s\oplus E^c\oplus E^s$,
$\dim(E^c)\leq 1$, all its periodic points have index $\operatorname{dim}(E^s\oplus E^c)$, and the central dynamics along the bundle $E^c$ is thin trapped for $f$.
One can thus apply~\cite[Theorem 13]{CP}
and deduce that $E^c$ is trivial. Hence $H(p)$ is hyperbolic.
\end{proof}
%
%

\subsection{Bound on the number of classes: proof of Corollary~\ref{c.centraldimension}}
We could argue as in \cite[section 6, page 724]{Cro08} but instead we give a different argument.
Let us assume by contradiction that there exists a countable collection of homoclinic classes
$H(p_n)$ having a same dominated splitting
$T_{H(p_n)}M=E\oplus E^c_1\oplus E^c_2\oplus F$, such that
$E^c_i$ are one-dimensional, and that there exists for each $n$
a periodic orbits $O_n^i$ homoclinically related to $p_n$ and whose Lyapunov exponent
along $E^c_i$ is arbitrarily weak.
One can assume that $H(p_n)$ converges towards a chain-transitive set $\Lambda$.
We claim that for $n$ large, the set $\Lambda$ is contained in the chain-unstable set of
$H(p_n)$. Arguing similarly, $\Lambda$ is contained in the chain-stable set of $H(p_n)$,
hence all the chain-recurrence class $H(p_n)$ contain $\Lambda$ and should coincide,
giving the contradiction.

Let us fix $\sigma\in (0,1)$ close to $1$. The domination and the fact that $O_n^2$ is weak along $E^c_2$
implies that there exists $x_n\in O_n^2$ that is a $(1,\sigma, E\oplus E^c_1)$-hyperbolic point for $f$. One can
assume that the sequence $(x_n)$ converges towards $x\in \Lambda$. One can also consider $y_n\in H(p_n)\cap
W^u(O_n^1)$ arbitrarily close to $x$. Since $O_n^1=\alpha(y_n)$ is weak along $E^c_1$, by
Lemma~\ref{l.hyperbolic-time}, there exists $C>0$ and one can replace the points $y_n$ and the point $x$, so
that the points $y_n$ are $(C,\sigma,E^c_2\oplus F)$-hyperbolic point for $f^{-1}$ and $x$ is $(1,\sigma,E\oplus
E^c_1)$-hyperbolic point for $f$. As a consequence, the unstable manifold of $y_n$ (and of $O_n^1\subset
H(p_n)$) meets the stable manifold of $x\in \Lambda$. This proves the claim and ends the proof.

\subsection{Index completeness: proof of Corollary~\ref{c.index-completeness}}

Let us consider $\Lambda$ as in the statement of Corollary~\ref{c.index-completeness}. By Theorem~\ref{main}, it
has a splitting $T_\Lambda M=E^s\oplus E^c_1\oplus\dots\oplus E^c_k\oplus E^u$. We fix $\delta>0$ and want to
prove the existence of a periodic orbit that is $\delta$-close to $\Lambda$ for the Hausdorff distance and of any
index in $\{\dim(E^s),\dots,\dim(E^s)+k\}$. We fix $U_0$ a neighborhood of $\Lambda$ contained in the
$\delta$-neighborhood $U$ of $\Lambda$.

In the case $k=0$, the set $\Lambda$ is hyperbolic and the shadowing lemma
shows that $\operatorname{ind}(\Lambda)=\dim(E^s)$.

In the case $k=1$ and all the invariant measure have a central Lyapunov exponent equal to zero,
then by Theorem~\ref{t.hausdorff} the set $\Lambda$ is the Hausdorff limit of a sequence of periodic orbits.
Their central Lyapunov exponent is arbitrarily close to zero, so by Lemma~\ref{l.index}
both indices $\dim(E^s)$ and $\dim(E^s)+1$ appear, as required.

In the remaining cases one has $k\geq 1$ and there exists an invariant measure whose
exponent along $E^c_1$ is non-zero.
Then Theorem~\ref{main2} and Lemma~\ref{l.index} imply that $\dim(E^s)$ is contained in $\operatorname{ind}(\Lambda)$.
Moreover, there exists a periodic point $p$ whose orbit is contained in $U_0$
and is $\delta$-close to $\Lambda$ for the Hausdorff distance,
whose local homoclinic class $H(p,U_0)$ contains $\Lambda$ and such that the index of $p$
is equal to $\dim(E^s)$ or $\dim(E^s)+1$.

Similarly, $\dim(E^s)+k\in \operatorname{ind}(\Lambda)$,
and there exists a periodic point $q$
whose orbit is contained in $U_0$,
whose local homoclinic class $H(q,U_0)$ contains $\Lambda$ and such that the index of $q$
is equal to $\dim(E^s)+k$ or $\dim(E^s)+k-1$.

Let us choose any $i\in \{\operatorname{ind}(p),\dots, \operatorname{ind}(q)\}$.
By Corollary~\ref{c.class}, since $H(p,U_0)$ and $H(q,U_0)$ intersect,
for any neighborhood $U_1\subset U$ of $\overline{U_0}$, the class
$H(p,U_1)$ contains $q$. Then, by Theorem~\ref{indices} and
Remark~\ref{r.indices}, one can choose $U_2\subset U$ containing $U_1$
such that $H(p,U_2)$ contains a hyperbolic periodic point $z$ of index $i$.
By Corollary~\ref{c.class}, the class $H(z,U)$ contains $p$
and a periodic point $x$ arbitrarily close to $p$,
whose orbit is contained in $U$ and of any index $i$.
By construction this orbit is $\delta$-close to $\Lambda$ for the Hausdorff
distance. Hence $\Lambda$ is index complete.

The end of the proof follows.
\bigskip

\subsection*{Acknowledgements.}
S.C was partially supported by the ANR project \emph{DynNonHyp} BLAN08-2 313375
and D.Y. by NSFC 11001101 and Ministry of Education of P. R. China 20100061120098
and the IMB of the Universit\'e of Bourgogne.
We are gratefull to the Universit\'e Paris 13 (and the LAGA), the ICTP,
Pekin University and the IHES for their hospitality.
We also thank C. Bonatti, S. Gan and L. Wen for interesting discussions
related to this work.

\vskip 5pt

\noindent Sylvain Crovisier

\noindent CNRS - Laboratoire Analyse, G\'eom\'etrie et Applications, UMR 7539,

\noindent Institut Galil\'ee, Universit\'e Paris 13, 99 Avenue J.-B. Cl\'ement, 93430 Villetaneuse, France.

\noindent crovisie@math.univ-paris13.fr

\vskip 5pt

\noindent Martin Sambarino

\noindent CMAT-Facultad de Ciencias, Univ. de la Rep\'{u}blica,

\noindent Igu\'{a} 4225, Montevideo 11400, Uruguay.

\noindent samba@cmat.edu.uy

\vskip 5pt

\noindent Dawei Yang

\noindent School of Mathematics, Jilin University, Changchun, 130012, P.R. China

\noindent yangdw1981@gmail.com

\end{document}